\documentclass{amsart}
\usepackage{xspace}
\usepackage{bbm}
\usepackage{graphicx}
\usepackage{subcaption}
\usepackage{titlesec}

\newcommand{\bbN}{\mathbb{N}}
\newcommand{\bbR}{\mathbb{R}}
\newcommand{\bbC}{\mathbb{C}}
\newcommand{\bbCP}{\mathbb{C}\mathrm{P}}
\newcommand{\bbZ}{\mathbb{Z}}
\newcommand{\imi}{\mathbbm{i}}

\newcommand{\matSL}{\mathrm{SL}}
\newcommand{\matSO}{\mathrm{SO}}
\newcommand{\matsl}{\mathrm{sl}}
\newcommand{\matSU}{\mathrm{SU}}

\newcommand{\disk}{\mathcal{D}}
\newcommand{\Pone}{\mathrm{P}^1}

\newcommand{\Sone}{\mathrm{S}^1}
\newcommand{\Stwo}{\mathrm{S}^2}
\newcommand{\Sthree}{\mathrm{S}^3}
\newcommand{\spaceS}{\mathrm{S}^3}
\newcommand{\spaceH}{\mathrm{H}^3}
\newcommand{\spaceADS}{\mathrm{AdS}_3}
\newcommand{\spaceDS}{\mathrm{dS}_3}
\newcommand{\loopgroup}{\Lambda^\ast}

\newcommand{\id}{\mathbbm{1}}
\newcommand{\deriv}{\mathrm{d}}
\newcommand{\half}{\tfrac{1}{2}}
\newcommand{\ol}{\overline}
\newcommand{\suchthat}{\mid}
\newcommand{\transpose}[1]{{#1}^\mathrm{t}}
\newcommand{\abs}[1]{|#1|}
\newcommand{\DOT}[2]{\langle#1,\,#2\rangle}
\DeclareMathOperator{\diag}{diag}
\DeclareMathOperator{\tr}{tr}
\DeclareMathOperator*{\res}{res}
\DeclareMathOperator*{\qres}{qres}
\DeclareMathOperator{\sign}{sign}
\newcommand{\Div}{\mathrm{div}_{\Sone}}

\newcommand{\zbar}{\bar{z}}

\newcommand{\sym}{evaluation\xspace}

\newcommand{\DPW}{GWR\xspace}

\newtheorem{theorem}{Theorem}[section]
\newtheorem{lemma}[theorem]{Lemma}
\newtheorem{proposition}[theorem]{Proposition}
\newtheorem{conjecture}[theorem]{Conjecture}

\theoremstyle{definition}
\newtheorem{remark}[theorem]{Remark}
\newtheorem{definition}[theorem]{Definition}
\newtheorem{example}[theorem]{Example}

\numberwithin{equation}{section}
\DeclareMathOperator{\dress}{\#}

\titleformat{\section}
  {\normalfont\fontsize{12}{17}\sffamily\bfseries}
  {\thesection}
  {1em}
  {}

\titleformat{\subsection}
  {\normalfont\fontsize{10}{14}\sffamily\bfseries}
  {\thesubsection}
  {1em}
  {}

\titleformat{\subsubsection}
  {\normalfont\fontsize{10}{14}\sffamily}
  {\thesubsubsection}
  {1em}
  {}

\title{Minimal $n$-Noids in hyperbolic and anti-de~Sitter 3-space}
\author{Alexander I. Bobenko}
\address{Institut f\"ur Mathematik, TU Berlin,
Str. des 17. Juni 136, 10623 Berlin, Germany}
\email{bobenko@math.tu-berlin.de}
\author{Sebastian Heller}
\address{Fachbereich Mathematik,
Universit\"at Hamburg, 20146 Hamburg, Germany} 
\email{seb.heller@gmail.com}
\author{Nick Schmitt}
\address{Institut f\"ur Mathematik, TU Berlin,
Str. des 17. Juni 136, 10623 Berlin, Germany}
\email{schmitt@math.tu-berlin.de}

\date{\today}
\begin{document}

\begin{abstract}
We construct minimal surfaces in hyperbolic  and
anti-de~Sitter 3-space with the topology of
a $n$-punctured sphere by loop group factorization methods. The end behavior of the surfaces is based on the asymptotics  of Delaunay-type surfaces, i.e., rotational symmetric minimal cylinders. The minimal surfaces in $\spaceH$ extend to Willmore surfaces in the conformal 3-sphere $S^3=\spaceH\cup\Stwo\cup\spaceH$.
\end{abstract}

\maketitle

\section*{Introduction}
\label{sec:intro}

\begin{figure}[b]
  \centering
  \includegraphics[width=0.475\textwidth]{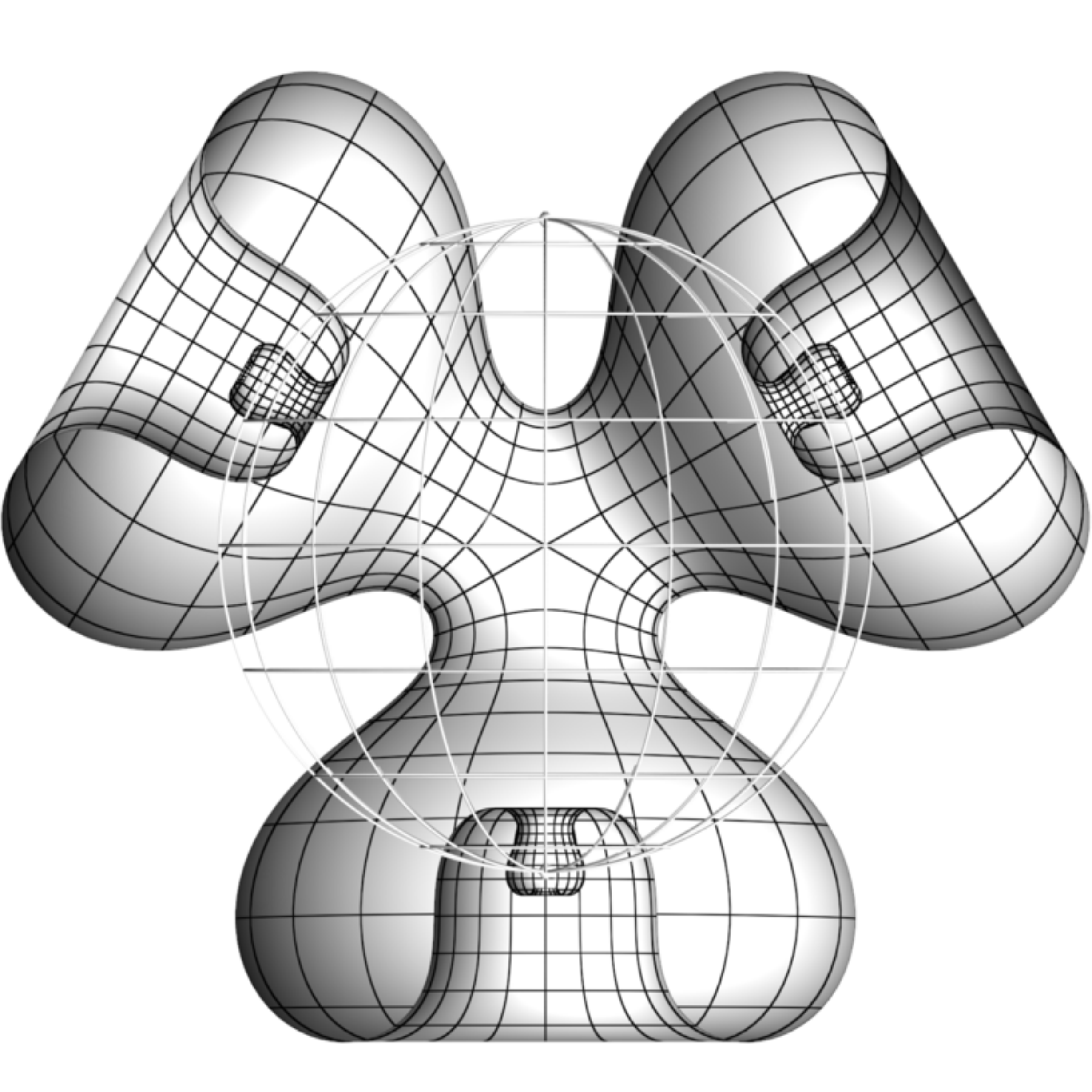}
  \includegraphics[width=0.475\textwidth]{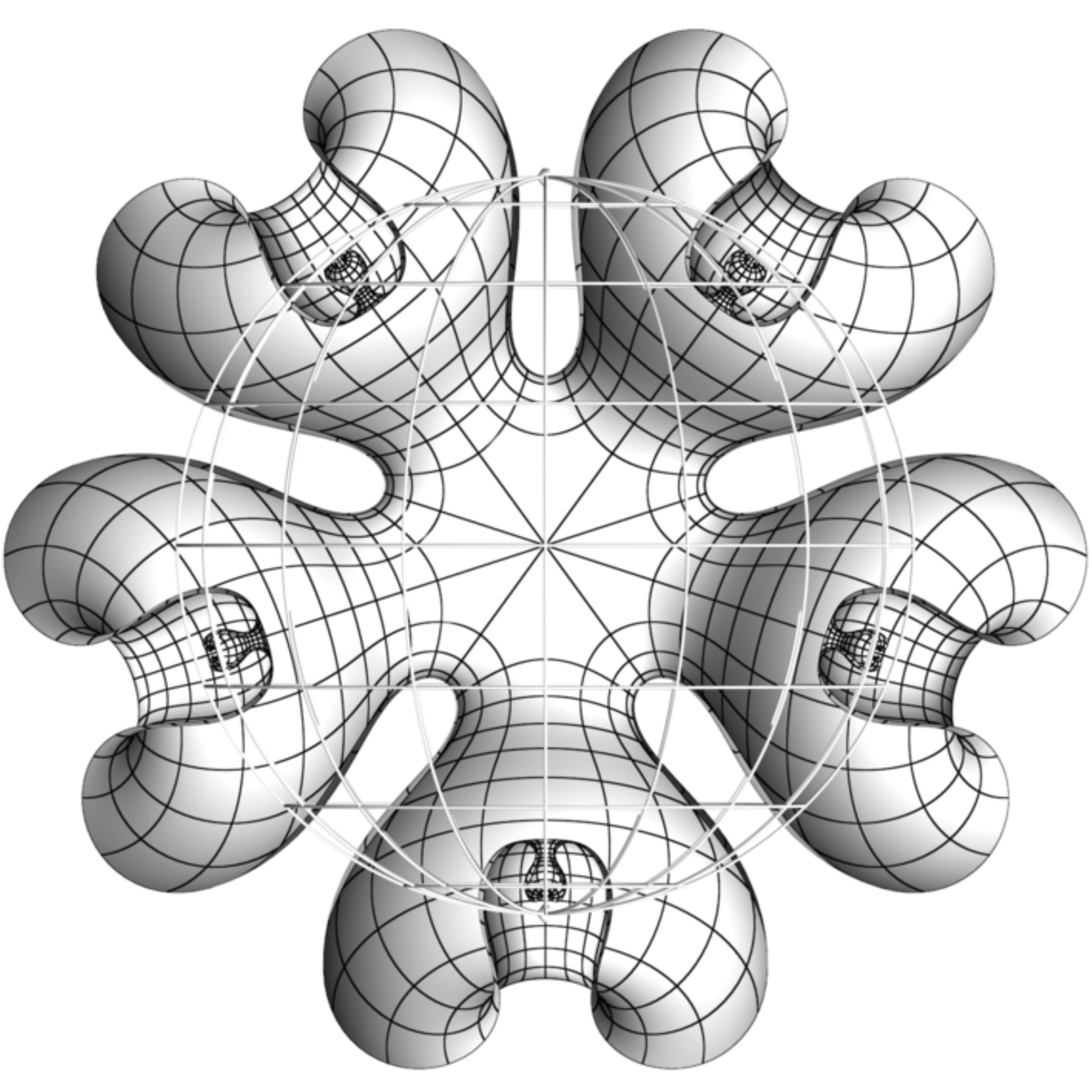}
  \caption{\small
    Cutaway views of equilateral minimal noids in $\spaceH\cup\Stwo\cup\spaceH$
    with end counts and cyclic symmetry orders $3$ and $5$.
    In all figures, the surface is stereographically projected to $\bbR^3$,
    the wireframe designates the ideal boundary,
    and the lines on the surface are curvature lines.}
  \label{fig:noid-h3}
\end{figure}

The AdS/CFT correspondence predicts that the physics of the
gravitational theory of anti-de~Sitter (AdS) spacetime is
equivalent to the physics of conformal field theory (CFT) on the
boundary of that spacetime~\cite{Maldacena_1998}. A particularly
important instance is the computation of the Wilson loop expectation
value, which by work of Maldacena is given by the (regularized) area
of a spacelike minimal surface in AdS spacetime with the loop as
boundary~\cite{Drukker_Gross_Ooguri_1999,Alday_Maldacena_2007,Alday_Maldacena_2009,Rey_Yee_2001}.
The aim of this paper is to provide new examples of minimal surfaces
in anti-de~Sitter spaces of non-trivial topological type and with
several boundary components.

It is well known that the equations for minimal surfaces in
anti-de~Sitter space are related to Hitchin self-duality
equations~\cite{Hitchin_1987,Alday_Maldacena_2009}.  In particular,
minimal surfaces in $\spaceADS$ are given by solutions corresponding
to points in the Hitchin component for rank 2, and minimal surfaces in
a totally geodesic $\spaceH\subset \text{AdS}_4$ are given by rank 2
solutions of the self-duality equations with nilpotent Higgs field.
As those, they fall into the class of integrable PDEs, and there are
powerful tools for computing large classes of
examples~\cite{Sakai_Satoh_2010,Ogata_2017,Dorfmeister_Inoguchi_Kobayashi_2014}. On
the other hand, it is hard to construct surfaces with non-trivial
finitely generated topology which are complete, i.e., the surface can
be continued to the boundary at infinity of the anti-de~Sitter space,
and the intersection is a finite number of topological circles. It is
worth remarking that surfaces given by global solutions of the
self-duality equations are not of that much interest in the AdS/CFT
correspondence as the extrinsic monodromy is always non-trivial and
the {\em intersection} with the boundary at infinity gets very
complicated.  By finite gap integration, cylindrical solutions have
been constructed; see for
example~\cite{Bakas_Pastras_2016}. In~\cite{Fonda_Giomi_Salvio_Tonni_2015}
a detailed numerical study for surfaces bounded by a finite number of
special curves (including circles, (super)ellipses and boundaries of
spherocylinders) is carried out, and the holographic entanglement
entropy and the holographic mutual information for those entangling
curves have been numerically computed.  

For the construction of meaningful
examples via loop group factorization methods, two main problems need
to be solved: The first is the proof of existence of potentials
depending on a loop parameter on a surface with non-trivial monodromy
which satisfy a certain reality condition (see remark~\ref{rem:closing}).
This reality condition is given explicitly only by
solving ODEs along non-trivial curves. We solve this problem for a
large class of examples. The second problem is concerned with the
construction of the minimal surfaces from potentials satisfying the
reality conditions: this problem is based on the fact that the
generalized Iwasawa factorization is not global, i.e., there exit
loops which do not admit an Iwasawa factorization
(remark~\ref{rem:iwasawa}).
It is hard to determine the curves on the surface
along which the factorization breaks down, and to characterize the
behavior of the minimal surface along those curves in general. On the
other hand, under a mild assumption on the degeneracy of the Iwasawa
decomposition, the minimal surface intersects the boundary at infinity
transversally; see for example~\cite[section 5]{Heller_Heller_2018}
for technical details and figures~\ref{fig:delaunay-h3} and
\ref{fig:noid-h3} for visualizations.  It is this failure of the
global Iwasawa decomposition which allows us to produce minimal
surfaces in anti-de~Sitter spaces with non-trivial (finitely generated)
topology and predicted topological intersections with the boundary at
infinity, at least numerically. We plan to investigate the remaining
theoretical questions concerning the intersections at infinity in
forthcoming work.

The structure of the paper is as follows: In section~\ref{sec:CMCDPW}
we briefly describe a unified loop group approach for constant mean
curvature surfaces in the symmetric spaces $\spaceS$, $\spaceADS$,
$\spaceH\cup\spaceH$ and $\spaceDS$, including the case of minimal
surfaces. In section~\ref{sec:H3} we first recall the conformal
surface geometry in the lightcone model, with special emphasis on
minimal surfaces in hyperbolic space.  We discuss some simple examples
of minimal surfaces in $\spaceH$: the hyperbolic disk as the
counterpart of the round sphere, and minimal Delaunay cylinders. We
then define $n$-noids and open $n$-noids, motivated by the behavior of
Delaunay cylinders at their ends. We prove the existence of open
$n$-noids, and conjecture that those surfaces actually give rise to
$n$-noids in the strict sense. In section~\ref{sec:ads3}, we explain
how to modify the techniques of section~\ref{sec:H3} in order to
obtain minimal surfaces in $\spaceADS$.  In section~\ref{sec:g3noid}
we study surfaces with three ends more explicitly
via the generalized Weierstrass representation (\DPW); these
investigations have enabled us to perform computer experiments.
Among others we construct a family of equilateral trinoids in $\spaceH$
with mean curvature $\abs{H}< 1$.
The paper is supplemented by figures visualizing global properties of
minimal surfaces in $\spaceH$, $\spaceDS$ and $\spaceADS$.
\section*{Acknowledgements} The first author is partially supported by the DFG Collaborative Research Center TRR 109 {\em Discretization in Geometry and Dynamics}.
The second author is supported by RTG 1670 {\em Mathematics inspired by string theory and quantum field theory}
 funded by the DFG. The third author is supported by the DFG Collaborative Research Center TRR 109 
{\em Discretization in Geometry and Dynamics}.
\section{The loop group method for CMC surfaces in symmetric spaces}
\label{sec:CMCDPW}

\subsection{Unitary frames}
\label{sec:unitary-frame}

  To construct constant mean curvature (CMC) surfaces in the symmetric spaces
  $\spaceS$, $\spaceADS$, $\spaceH\cup\spaceH$ and $\spaceDS$
  we use the following matrix models
  in $\matSL_2(\bbC)$:
\begin{equation}
  \label{eq:spaceform}
  \begin{array}{l|l|l}
    \text{space} & \text{matrix model} & \hphantom{-}\DOT{x}{y}\\
    \hline
    \spaceS &
    \matSU_2 &
    \hphantom{-}\half\tr x\hat{y}
    \\
    \spaceADS &
    \matSU_{11} &
    -\half \tr x\hat{y}
    \\
    \spaceH\cup\spaceH &
    \{X \in \matSL_2(\bbC) \suchthat {\transpose{\ol{X}}} = X\} &
    -\half\tr x\hat{y}
    \\
    \spaceDS &
    \{X \in \matSL_2(\bbC) \suchthat
        {\transpose{\ol{X}}} = e_0 X e_0 ^{-1}\} &
        \hphantom{-}\half\tr x\hat{y}
  \end{array}
\end{equation}
where
$e_0 = \diag(\imi,\,-\imi)$.
The third column of the table specifies
the inner product on $\mathrm{M}_{2\times 2}(\bbC)$
extending the metric on the symmetric space,
with sign chosen so that the signature of the
tangent space is $(\pm,\,+\,+)$.
Here $\hat y=\bigl[\begin{smallmatrix}d & -b\\-c & a\end{smallmatrix}\bigr]$
for $y=\bigl[\begin{smallmatrix}a & b\\c & d\end{smallmatrix}\bigr]$.


To use integrable systems methods we introduce the loop group
$\Lambda\matSL_2(\bbC)$ of real analytic maps $\Sone\to\matSL_2(\bbC)$,
and the subgroup $\Lambda_+\matSL_2(\bbC)$ of
loops which extend holomorphically to the interior of the unit disk.
The four involutions of $\Lambda\matSL_2(\bbC)$
\begin{equation}
  \label{eq:real-form}
  \begin{array}{l|l}
    \spaceS & X^\ast(\lambda) = {\transpose{\ol{ X(1/\ol{\lambda}) }}}^{-1}\\
    \spaceADS & X^\ast(\lambda) = e_0 {\transpose{\ol{ X(1/\ol{\lambda}) }}}^{-1} e_0^{-1}\\
    \spaceH & X^\ast(\lambda) = {\transpose{\ol{ X(-1/\ol{\lambda}) }}}^{-1}\\
    \spaceDS  & X^\ast(\lambda) = e_0 {\transpose{\ol{ X(-1/\ol{\lambda}) }}}^{-1}e_0 ^{-1}
  \end{array}
\end{equation}
determine four real forms $\{X\in\Lambda\matSL_2(\bbC) \suchthat
X^\ast = X\}\subset\Lambda\matSL_2(\bbC)$.  By an abuse of
terminology, we call elements of such a subgroup \emph{unitary}, or
emphasizing the involution e.g. $\spaceH$-unitary.  We will also
denote by $^\ast$ the corresponding involutions of the Lie algebra
$\Lambda\matsl_2(\bbC)$.

Given a choice of one of the real forms induced by~\eqref{eq:real-form}, a
\emph{unitary connection} is a $\Lambda\matsl_2(\bbC)$-valued $1$-form
$\eta$ on a Riemann surface $\Sigma$ with the following properties:
\begin{itemize}
\item
  $\eta$ is flat for all $\lambda$.
\item
  $\eta = \eta^\ast$.
\item
  $\eta$ has a simple pole at $\lambda=0$,
  $\eta_{-1} :=\res_{\lambda=0}\eta$
  has no $(0,\,1)$ part, $\det\eta_{-1} = 0$, and
  $\DOT{\eta_{-1}}{\eta_{-1}^\ast} \ne 0$.
\end{itemize}
A \emph{unitary frame} $F$ is a unitary solution to the ODE $\deriv F = F\eta$.

The \emph{\sym formula}  maps a unitary frame to the symmetric space:
\begin{equation}
  \label{eq:sym}
  f = \left.F\right|_{\lambda_0}\left.F^{-1}\right|_{\lambda_1}
\end{equation}
where $\lambda_0,\,\lambda_1\in\bbC^\ast$
are the \emph{\sym points} as follows:
\begin{equation}
  \label{eq:sym-point}
  \begin{array}{l|l|l}
    \text{space} & \text{\sym points} & \text{mean curvature $H$}\\
    \hline
    \spaceS\text{ and } \spaceADS& \lambda_0,\,\lambda_1 \in\Sone \text{ distinct}
    & \frac{\imi(\lambda_1+\lambda_0)}{\lambda_1 - \lambda_0}\\
    \spaceH \text{ and }\spaceDS
    & \lambda_0,\,\lambda_1\in\bbC^\ast \text{ with }\lambda_0\ol{\lambda_1} = -1
    & \frac{\lambda_1+\lambda_0}{\lambda_1 - \lambda_0}
  \end{array}
\end{equation}
The third column of the table lists
the mean curvature $H$
of the induced immersion $f$ up to sign,
derived in theorem~\ref{thm:unitary-frame}.
Note that $H$ satisfies
$\abs{H}<1$ for the symmetric spaces $\spaceH$ and $\spaceDS$. Formula \eqref{eq:sym} was first derived in \cite{Bobenko_1991} for CMC surfaces in $\spaceH$ and $\spaceS.$

\begin{theorem}
  \label{thm:unitary-frame}
  Let $\eta$ be a unitary connection on a Riemann surface $\Sigma$ with
  respect to one of the four real forms~\eqref{eq:real-form}.
  Let $F$ be a unitary frame satisfying $\deriv F = F \eta$.
  Then the \sym formula~\eqref{eq:sym} evaluated
  at \sym points~\eqref{eq:sym-point}
  yields a spacelike conformal CMC immersion $f$ into the
  corresponding symmetric space
  with metric
  \begin{equation}
    \label{eq:metric}
    v^2\deriv z\otimes\deriv\zbar,\quad
    v^2 = 2(\lambda_0^{-1}-\lambda_1^{-1})(\lambda_0-\lambda_1)
    \DOT{\alpha}{\alpha^\ast}
  \end{equation}
  and constant mean curvature as in~\eqref{eq:sym-point}.
\end{theorem}

\begin{remark}\label{rem:assofami}
The unitary connection $\eta$ is usually referred to as the associated
family of flat connections of the surface $f$.
\end{remark}

\begin{proof}
  We prove the theorem for $\spaceADS$;
  the proof for the other symmetric spaces is similar
  with sign changes.
  The unitary connection decomposes as
  \begin{equation}
    \label{eq:unitary-connection}
    \eta = (\alpha\lambda^{-1} + \beta)\deriv z +
    (\beta^\ast + \alpha^\ast\lambda)\deriv\zbar.
  \end{equation}
  The flatness of $\eta$ is equivalent to
  \begin{equation}
    [\alpha,\,\beta^\ast] = \alpha_{\zbar}
    ,\quad
    [\alpha,\,\alpha^\ast] + [\beta,\,\beta^\ast] = \beta_{\zbar} - {(\beta^\ast)}_z
    ,\quad
    [\beta,\,\alpha^\ast] = {(\alpha^\ast)}_z.
  \end{equation}

  Let $F_0 = \left.F\right|_{\lambda_0}$
  and $F_1 = \left.F\right|_{\lambda_1}$.
  To compute the metric
  \begin{equation}
  f_z = (\lambda_0^{-1}-\lambda_1^{-1})F_0 \alpha F_1^{-1},\quad
  f_{\zbar} = (\lambda_0-\lambda_1)F_0 \alpha^\ast F_1^{-1}.
  \end{equation}
  Since $\DOT{f_z}{f_z} = \DOT{f_{\zbar}}{f_{\zbar}} = 0$ the metric is
  $v^2 \deriv z\otimes \deriv\zbar$ with
  \begin{equation}
    v^2 = 2\DOT{f_z}{f_{\zbar}} = 2(\lambda_0^{-1} - \lambda_1^{-1})
    (\lambda_0 - \lambda_1)\DOT{\alpha}{\alpha^\ast}.
  \end{equation}
  Since $\lambda_0\ne\lambda_1$ and $\DOT{\alpha}{\alpha^\ast}\ne 0$,
  the metric is nonzero, so the \sym formula induces
  a conformal immersion.

  The normal is $N = F_0\gamma F_1^{-1}$ where
  \begin{equation}
    \gamma = -\frac{\imi}{2}\frac{[\alpha,\,\alpha^\ast]}{\DOT{\alpha}{\alpha^\ast}}.
  \end{equation}
  Using flatness,
  \begin{equation}
    f_{z\zbar} = (\lambda_0^{-1}-\lambda_1^{-1})
    F_0 (\lambda_0 \alpha^\ast\alpha -\lambda_1\alpha\alpha^\ast) F_1^{-1}.
  \end{equation}
  Using that $\DOT{\alpha\alpha^\ast + \alpha^\ast\alpha}{\gamma} = 0$,
  the mean curvature of $f$ is
  \begin{equation}
    H = 2 v^{-2}\DOT{f_{z\zbar}}{N} =
    -\half v^2(\lambda_0^{-1}-\lambda_1^{-1})(\lambda_0 + \lambda_1)
    \DOT{[\alpha,\,\alpha^\ast]}{\gamma} =
    \imi \frac{\lambda_1 + \lambda_0}{\lambda_1 - \lambda_0}.
    \qedhere
  \end{equation}
\end{proof}

\begin{remark}
  With other choices of \sym formulas and \sym points,
  CMC surfaces can also be constructed from unitary connections
  in the symmetric spaces related by the Lawson correspondence:
  \begin{align}
    \text{corresponding to $\spaceS$}&:\
    \quad\text{$\spaceH$ ($\abs{H}>1$) and $\bbR^3$}\\
    \text{corresponding to $\spaceADS$}&:
    \quad\text{$\spaceDS$ ($\abs{H}>1$) and $\bbR^{2,1}$}.
  \end{align}
\end{remark}

\begin{remark}\label{rem:sine}
  In the case the Hopf differential is $\deriv z^2$,
  after a coordinate change the flatness of the unitary connection
  is Gauss equation on the metric $v^2 = e^{2u}$:
  \begin{align}
    \label{eq:gauss-equation}
    \spaceS &:\ \Delta u + 2\sinh 2u = 0
    &
    \spaceH\ (\abs{H}<1)&:\ \Delta u - 2\cosh 2u = 0\\
    \spaceADS &:\ \Delta u - 2\sinh 2u = 0
    &
    \spaceDS\ (\abs{H}<1)&:\ \Delta u + 2\cosh 2u = 0.
  \end{align}
\end{remark}

\subsection{Holomorphic frames}
\label{sec:gwr}
To construct CMC immersions by theorem~\ref{thm:unitary-frame}, one is
required to produce a flat unitary connection.  The flatness is the
Gauss equation, a partial differential equation on the metric.  This
section introduces the generalized Weierstrass representation
(\DPW)~\cite{Dorfmeister_Pedit_Wu_1998}. In this construction, the PDE
is replaced by an ordinary differential equation together with a
Iwasawa loop group factorization.  For a more comprehensive treatment
of real forms of loop groups see~\cite{Kobayashi_2011}.  The case
$\bbR^{2,1}$ was considered in~\cite{Brander_Rossman_Schmitt_2010}.

A \emph{\DPW potential} $\xi$ is a $\Lambda\matsl_2(\bbC)$-valued
$(1,\,0)$-form on a Riemann surface $\Sigma$ satisfying the following
condition: $\xi$ has a simple pole at $\lambda=0$, and
$\xi_{-1}:=\res_{\lambda=0}\xi$ satisfies $\det\xi_{-1} = 0$ and is
nowhere zero.  A \emph{\DPW frame} $\Phi$ for $\xi$ is a holomorphic
map from the domain to $\Lambda\matsl_2(\bbC)$ satisfying
$\deriv\Phi = \Phi\xi$.  Choosing a real form, an
\emph{Iwasawa factorization} of $\Phi$ is
\begin{equation}
  \label{eq:iwasawa}
  \Phi = F B,\quad
  F^\ast = F\text{ and }B\in\Lambda_+\matSL_2(\bbC).
\end{equation}

The Iwasawa factorization can be computed via the Birkhoff
factorization as follows.  When in the big cell,
${\Phi^\ast}^{-1}\Phi$ has a Birkhoff factorization
\begin{equation}
  \label{eq:birkhoff}
        {\Phi^\ast}^{-1}\Phi = {X_+^\ast}^{-1} X_+,\quad
        X_+\in\Lambda_+\matSL_2(\bbC).
\end{equation}
Then the desired Iwasawa factorization of $\Phi$ is
\begin{equation}
  \Phi = F B,\quad
  F:=\Phi X_+^{-1},\quad
  B:= X_+
\end{equation}
because $F^\ast = F$.

\begin{theorem}
  \label{thm:gwr}
  Let $\xi$ be a \DPW potential,
  and $\Phi$ a corresponding \DPW frame.
  If $\Phi$ has an Iwasawa factorization,
  then $F$ is a unitary frame.
  Hence by theorem~\ref{thm:unitary-frame} $F$
  induces a conformal CMC immersion.
\end{theorem}

\begin{proof}
  Since $\deriv \Phi = \Phi\xi$, $\deriv F = F\eta$, and $\Phi = F B$, then
  \begin{equation}
    \eta = \xi . (B^{-1})
  \end{equation}
  where the dot denotes the gauge action
  $\xi.g := g^{-1}\xi g + g^{-1}\deriv g$.
  Since $B\in\Lambda\matSL_2(\bbC)$ and $\xi$ has a simple pole in $\lambda$,
  then $\eta$ has a simple pole in $\lambda$.
  Since $\xi$ has no $(0,\,1)$ part, then $\eta_{-1} := \res_{\lambda=0}\eta$
  has no $(0,\,1)$ part.
  Since $F^\ast = F$, then $\eta^\ast=\eta$, hence $\eta$ is a unitary connection.
\end{proof}

\begin{remark}
  \label{rem:hopf}
  The Hopf differential of CMC immersion induced by a \DPW potential $\xi$
  is of the form $c(\lambda)Q\deriv z^2$ where $c$ is $z$-independent and
  $Q$ is the leading term (coefficient of $\lambda^{-1})$ of $\det \xi$.
\end{remark}

\begin{remark}
  \label{rem:iwasawa}
  In the case of the spaceform $\spaceS$, the \DPW frame always has an
  Iwasawa factorization.  For the other three real forms, the \DPW
  frame may fail to have an Iwasawa factorization, generally on some
  real analytic subset of the domain.  On this set the surface is
  singular, in many cases going to the ideal boundary.
\end{remark}

If the domain is not simply connected, $\Phi$ has monodromy and the
induced CMC immersion on the universal cover does not generally
\emph{close}, that is, descend to an immersion of the domain.  A
sufficient condition for closing is the following:
\begin{remark}
  \label{rem:closing}
  The induced CMC immersion closes if the monodromy $M$ of $\Phi$ is
  unitary (intrinsic closing), and at the \sym points $M(\lambda_0) =
  M(\lambda_1) \in\{\pm 1\}$ (extrinsic closing).
\end{remark}

\section{Minimal $n$-noids in hyperbolic 3-space}
\label{sec:H3}

We study minimal surfaces in hyperbolic 3-space $\spaceH$ which
intersect the boundary at infinity perpendicularly. A convenient setup
from a geometric point of view is conformal surface geometry in
the lightcone model of the 3-sphere.

\subsection{The lightcone model for $\spaceH$}
The lightcone approach to conformal surface geometry is classical;
for details we refer to
\cite{Burstall_Pedit_Pinkall_2002,Quintino_2009} and the references
therein.  We consider Minkowski space $V=\bbR^{4,1}$ with its
standard inner product $(\cdot,\,\cdot)$ inducing the quadratic form
\begin{equation}
  q(x_0, \dots, x_4) = -x_0^2 + x_1^2 + \dots + x_4^2.
\end{equation}
The lightcone
\begin{equation}
  \mathcal L=\{\bbR x\in PV\mid x\neq0,\; q(x)=0\}
\end{equation}
is diffeomorphic to the 3-sphere $\Sthree$ via
\begin{equation}
  (x_1, \dots ,x_4)\in \spaceS\ \mapsto\ \bbR(1, x_1, \dots ,x_4)\in\mathcal L.
\end{equation}
Thus $\mathcal L$ inherits a conformal structure from the round metric
on $\spaceS$. It is well known that the group $\matSO(4,1)$ acts on
$\mathcal L$ by conformal transformations. In fact,
\begin{equation}
  \matSO_+(4,1)=\{g\in \matSO(4,1)\mid (g(e_0),e_0)>0\}
\end{equation}
is the group of (orientation preserving) conformal diffeomorphisms of
$\spaceS$ (equipped with the round conformal structure).

\subsubsection{Hyperbolic 3-space}
\label{hyp3space}
Taking the  spacelike vector $v_\infty=e_4$
we obtain two copies of hyperbolic 3-space
as
\begin{equation}
  \mathcal L\setminus (\mathcal L\cap P (e_4^\perp)) =
  \{(x_0,x_1,x_2,x_3,-1)\mid -x_0^2+x_1^2+x_2^2+x_3^2=-1\}.
\end{equation}
This space is naturally equipped with the metric of constant curvature $-1$.
The subgroup $\matSO_+(3,1)\subset \matSO_+(4,1)$ (defined by fixing the
vector $e_4$) realizes the isometry group of hyperbolic 3-space.

\subsubsection{Surfaces in the lightcone model}
We consider conformal immersions
\begin{equation}
  f\colon \Sigma\to \Sthree
\end{equation}
from a Riemann surfaces $\Sigma$ into the conformal 3-sphere. The map
$f$ is equivalent (in conformal geometry) to the line bundle of
light-like vectors
\begin{equation}
  \sigma^\ast\mathcal L\to\Sigma
\end{equation}
for $\sigma=(1,f)$. The fact that $f$ is an immersion means that for
any (local) lift $\tilde \sigma=g \sigma$ and (local) pointwise
independent vector fields $X,Y$ on $\Sigma$
\begin{equation}
  \text{span}(\tilde\sigma,\,X\cdot\tilde\sigma,\,Y\cdot\tilde\sigma)
\end{equation}
is a (real) 3-dimensional bundle (where $\cdot$ denotes
the derivative). Conformality of $f$ means that for a local holomorphic
chart $z$ on $\Sigma$ we have
\begin{equation}
  (\tilde\sigma_z,\tilde\sigma_z)=0=(\tilde\sigma_{\bar z},\tilde\sigma_{\bar z})
\end{equation}
where $g_z:=\frac{\partial}{\partial z}\cdot g$ and $g_{\bar
  z}:=\frac{\partial}{\partial \bar z }\cdot g$ for any
(vector-valued) function $g$.

A fundamental object in conformal surface theory is the mean curvature
sphere congruence. The mean curvature sphere is defined locally by
\begin{equation}
  \mathcal S=\text{span}
  (\tilde\sigma,\,\tilde\sigma_z,\,\sigma_{\bar z},\,\tilde\sigma_{z,\bar z}).
\end{equation}

\begin{proposition}
  The surface $f$ considered as a surface in hyperbolic 3-space
  defined by $v_\infty=e_4$ is of constant mean curvature $H$ if and
  only if
  \begin{equation}
    H=(e_4^\perp,\,e_4^\perp)
  \end{equation}
  is constant, where $(e_4^\perp)_p$ is the projection of $e_4$ to
  $\mathcal S_p$. Consequently, $f$ is minimal if and only if $e_4$ is
  contained (as a constant section) in $\mathcal S$.
\end{proposition}

In the following, we are interested in conformally parametrized
surfaces $f\colon\Sigma\to\spaceS$ into the conformal 3-sphere such
that its intersection with
\begin{equation}
  \spaceH\cup\spaceH \ =
  \ \spaceS\setminus \mathrm{S}^2 \ =
  \ \mathcal L\setminus (\mathcal L\cap P (e_4^\perp))
\end{equation}
is a minimal surface. In order to apply the \DPW approach we need an
explicit isometry
\begin{equation}
  \Psi\colon \mathcal L\setminus (\mathcal L\cap P (e_4^\perp))\to
  \{X \in \matSL_2(\bbC) \suchthat {\transpose{\ol{X}}} = X\}.
\end{equation}
This is provided by
\begin{equation}
  \label{eq:lightconematrix}
  \Psi([x_0,x_1,x_2,x_3,x_4])\mapsto \frac{1}{x_4}
  \begin{bmatrix}x_0+x_1& x_2+i x_3\\x_2-i x_3& x_0-x_1 \end{bmatrix}.
\end{equation}
Note that the two copies of $\Psi(\mathcal L\setminus (\mathcal L\cap
P (e_4^\perp)))$ are given by the sets of positive definite and
negative definite symmetric $\matSL_2$-matrices.

\subsection{Basic examples}
We first illustrate the \DPW approach for some basic surfaces.  Recall
that for $\spaceH$, the real involution on $\Lambda\matSL_2(\bbC)$ is
given by
\begin{equation}
  X^\ast(\lambda)=e_0\overline{X(-1/\bar\lambda)}^{-1}e_0^{-1},
\end{equation}
and unitary connections are flat connections of the form
\begin{equation}
  d+\eta(\lambda)=d+\eta_0+\lambda^{-1}\eta_{-1}+\lambda\eta_1
\end{equation}
with $\transpose{\ol{\eta_0}}=-\eta_0$ and
$\transpose{\ol{\eta_{-1}}}=\eta_1$ where $\eta_{-1}$ is a nowhere
vanishing $(1,0)$-form with values in the nilpotents.

\subsubsection{The sphere in $\spaceH$}
The simplest example of a \DPW potential is given by
\begin{equation}
  \xi(\lambda)=\begin{bmatrix}0 & \lambda^{-1} \\
  0 & 0\end{bmatrix}dz
\end{equation}
on the complex plane, with \DPW frame
\begin{equation}
  \Phi(\lambda)=\begin{bmatrix}1 & \lambda^{-1} z \\
  0 & 1\end{bmatrix}.
\end{equation}
The $\spaceH$ Iwasawa factorization is given by
\begin{equation}
  \Phi(\lambda)=F(\lambda)B(\lambda)=\frac{1}{\sqrt{1-z\bar z}}
  \begin{bmatrix}1 & \lambda^{-1} z \\
    \lambda \bar z & 1\end{bmatrix}\frac{1}{\sqrt{1-z\bar z}}
    \begin{bmatrix}1 & 0 \\
      -\lambda \bar z & 1-z\bar z\end{bmatrix}
\end{equation}
and taking $\lambda_0=1$ and $\lambda_{1}=-1$ we obtain
\begin{equation}
  \label{eq:H3sphere}
  f=F(1)F(-1)^{-1}=\frac{1}{1-z\bar z}
  \begin{bmatrix} 1+ z\bar z& 2z \\ 2\bar z &1+ z\bar z\end{bmatrix}.
\end{equation}
Restricting to the unit disk $D\subset\bbC$, this is just a
conformally parametrized totally geodesic hyperbolic disk inside
hyperbolic 3-space, with induced metric
\begin{equation}
  \frac{1}{1-z\bar z}\deriv z\otimes \deriv\bar z.
\end{equation}
Note that this example is rather special as we are able to write down
both the \DPW frame and its factorization in terms of elementary
functions.
\begin{remark}
  The surface $f$ has the same \DPW potential as the round minimal
  2-sphere in $\spaceS$, and serves as the simplest example of a
  minimal surface in $\spaceH$. On the other hand, as a map to
  $\spaceH$, $f$ is not well-defined on the whole plane $\bbC$ or
  projective line, but crosses the ideal boundary at $\infty$
  \begin{equation}
    \mathrm{S}^2_\infty=\mathcal L\cap P (e_4^\perp)
  \end{equation}
  along the unit circle $\Sone\subset\bbC$ where the Iwasawa
  decomposition breaks down. By~\eqref{eq:H3sphere} and~\eqref{eq:lightconematrix}
  (or by geometric reasoning) $f$ can be
  extended as a conformal surface into the conformal 3-sphere
  $\mathcal L$, a phenomena which turns out to be typical in the
  examples below.  In the case at hand, we obtain a conformally
  parametrized totally umbilic sphere
  \begin{equation}
    z\in\bbCP^1\mapsto
    [1+z\bar z,\, 0,\, z+\bar z,\, \imi (\bar z-z),\, 1-z\bar z]
    \in\mathcal L.
  \end{equation}
\end{remark}

\subsubsection{Delaunay cylinders in $\spaceH$}

\begin{figure}[b]
  \centering
  \begin{subfigure}[t]{0.49\textwidth}
    \centering
    \includegraphics[width=\textwidth]{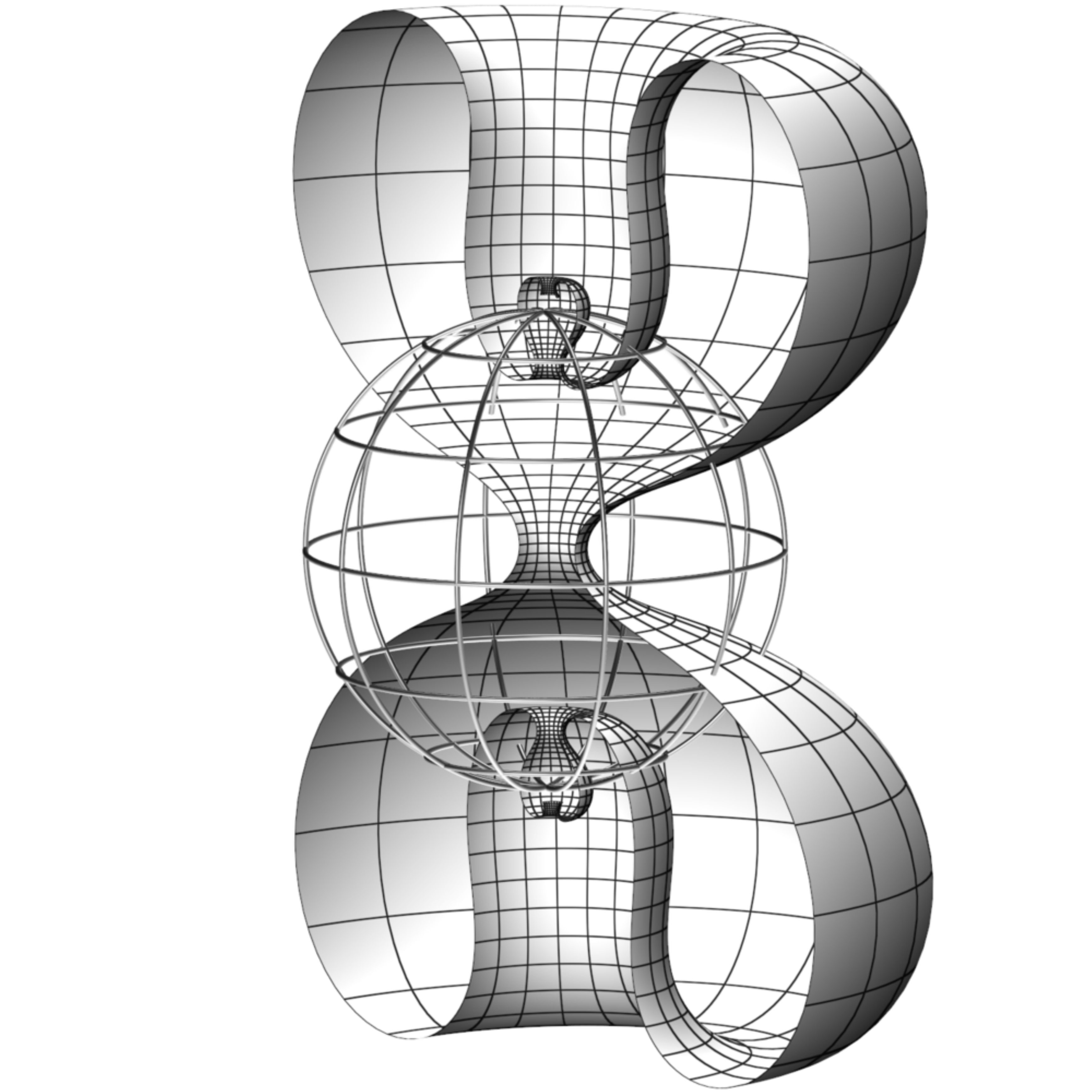}
    \caption{
      \small
      Delaunay surface in
      $\spaceH\cup\Stwo\cup\spaceH.$
      Each of the two ends of this
      surface of revolution oscillates between the two copies of
      $\spaceH$, crossing the ideal boundary infinitely often.
    }
    \label{fig:delaunay-h3}
  \end{subfigure}
  \begin{subfigure}[t]{0.49\textwidth}
    \centering
    \includegraphics[width=\textwidth]{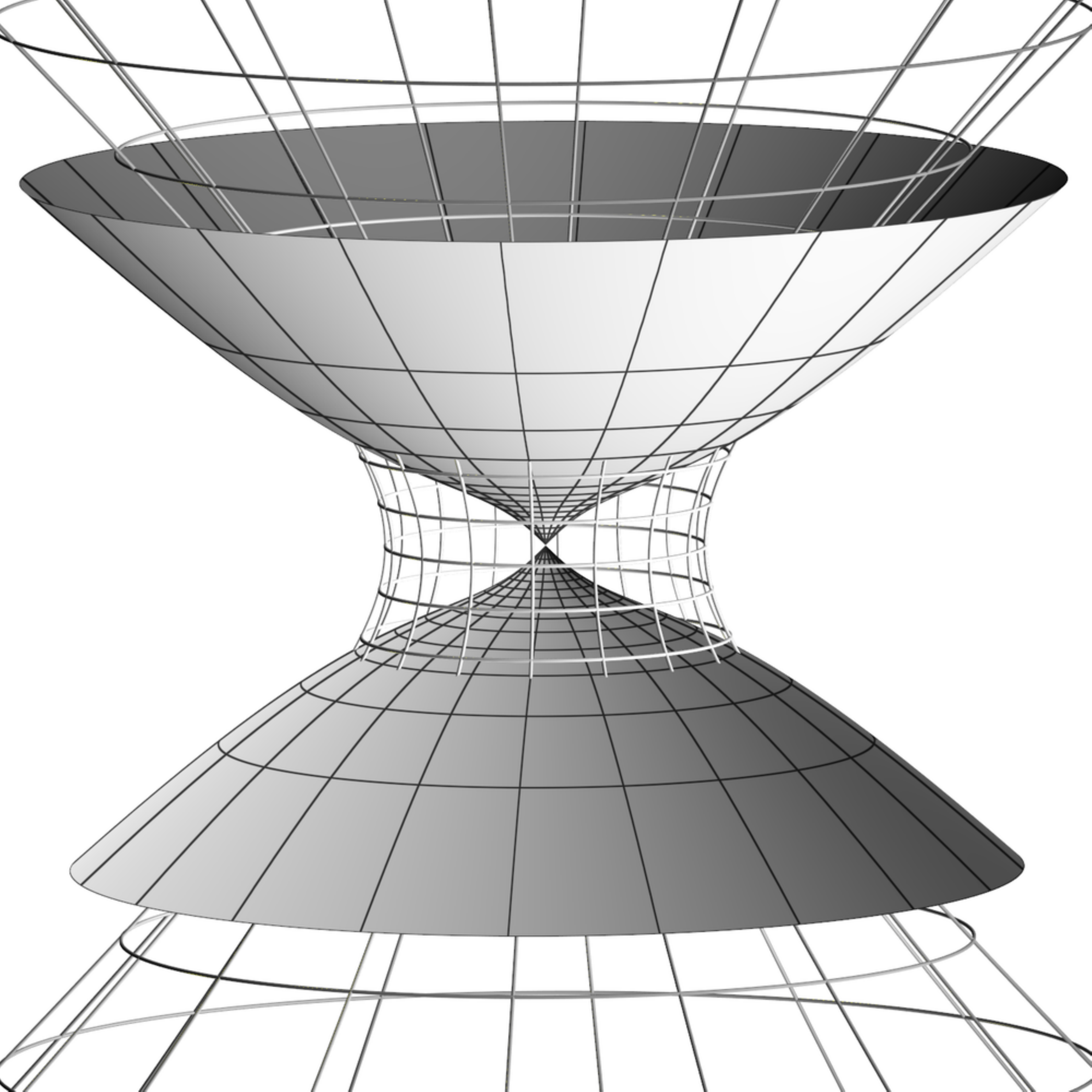}
    \caption{\small
      Minimal Delaunay surface in $\spaceADS.$
      At the cone point (center of image), the profile curve
      of this surface of revolution crosses the revolution axis and the
      surface fails to be immersed.
    }
    \label{fig:delaunay-ads3}
  \end{subfigure}
  \caption{}
  \label{fig:delaunay}
\end{figure}

A less trivial class of surfaces is given by Delaunay surfaces in
$\spaceH$.  Minimal Delaunay cylinders in $\spaceH$
(figure~\ref{fig:delaunay-h3}) were first described
in~\cite{Babich_Bobenko_1993} in terms of their elliptic spectral
data.  They come in a real 1-dimensional family of geometrically
distinct surfaces.  On $\bbC/(2\pi i\bbZ)$, the Hopf
differential of a Delaunay cylinder is a constant multiple of
$(\deriv w)^2$, and the conformal factor is a solution of the
cosh-Gordon equation (remark~\ref{rem:sine}).
The conformal factor can be given
explicitly in terms of the Weierstrass $\wp$-function on a rectangular
elliptic curve; see for example
\cite{Babich_Bobenko_1993,Bakas_Pastras_2016}. The surface is
rotational symmetric, and the conformal factor only depends on one
(real) variable and is periodic --- but blows up once each period where
the surface intersects the ideal boundary at $\infty$
(figure~\ref{fig:delaunay-h3}).

We consider the
Delaunay cylinders parametrized on the two-punctured sphere $\bbC^\ast$,
whose Hopf differential is a constant multiple of
$(\deriv z)^2/z^2$.
To construct this 1-dimensional family on the domain $\bbC^\ast$ via
the \DPW approach take \sym points
\begin{equation}
  \lambda_0=-\lambda_1=1
\end{equation} determined by the mean curvature $H=0$ in the
last column of table~\eqref{eq:sym-point}.  The \DPW potential on the
domain $\bbC^\ast$ is $\imi A\deriv z/z$ where $A$ is a
$z$-independent loop with the following properties (remark~\ref{rem:closing}):
\begin{itemize}
\item
  $q\in\bbR^\ast$ is a branch-point of the spectral curve: $\det A(q) = 0$;
\item
  intrinsic closing condition: $A = A^\ast$; see the third row
  of~\eqref{eq:real-form};
\item
  extrinsic closing condition: eigenvalues of $A(\lambda_0)$ are $\pm\imi/2$.
\end{itemize}
It follows that the eigenvalues of the frame monodromy around the puncture $z=0$
are $\exp(\pm 2\pi \nu)$ where \begin{equation}
  \nu= \tfrac{\imi}{2}\sqrt{
    \tfrac{(\lambda-q)(-\lambda^{-1} - q)}
          {q^2-1}}.
\end{equation}
%
More explicitly, for  $\spaceH$ we may take $A$ to be
\begin{align}
  \quad
  \tfrac{1}{2\sqrt{q^2-1}}
  \begin{bmatrix}0 & \lambda^{-1} + q\\ \lambda - q & 0\end{bmatrix}
\end{align}
constrained by the condition
that the term under the square root is positive, i.e., $\abs{q}>1$.

The \DPW frame $\Phi$ is based at $z=1$, i.e., $\Phi(1) = \id$. Hence
$\Phi = \exp(\imi A \log z)$.

\begin{theorem}
  The \DPW construction applied to the above data gives Delaunay
  cylinders.
\end{theorem}

\begin{proof}
  The \DPW frame for the potential $\xi$ is $\Phi = \exp(\imi A \log
  z)$.  The monodromy of $\Phi$ around the puncture $z=0$ is $M =
  \exp(2\pi A)$, satisfying the intrinsic closing condition $M^\ast =
  M$ due to the symmetry of $A$, and the extrinsic closing condition
  (on the cylinder) $M(\lambda_0) = M(\lambda_1) = -\id$ due to the
  fact that the eigenvalues of $A$ are $\pm\imi/2$ at each \sym point.
  Hence the surface closes on $\bbC^\ast$.

  To show that the induced surface is a surface of revolution, changing
  coordinates $x+\imi y = \imi \log z$, then $\Phi = \exp( (x+\imi y)A)
  = \exp(x A)\exp(\imi y A)$.  Since $\exp(x A)$ is unitary, then the
  unitary factor in the Iwasawa decomposition of $\Phi$ is $F(x,\,y) =
  \exp(x A) G(y)$, where $G$ is the unitary factor of $\exp(\imi y A)$.
  Hence $F$ is equivariant.  Hence the surface $F_{\lambda_0}F^{-1}_{\lambda_1}$
  is equivariant with equivariant action
  on the profile curve $X = X(y) = G_{\lambda_0}G^{-1}_{\lambda_1}$
  given by
  \begin{equation}
    X \mapsto \exp(x A(\lambda_0)) X \exp(-x A(\lambda_1)).
  \end{equation}
  This action has closed orbits with period
  $x\in[0,\,2\pi]$ because the eigenvalues of $A(\lambda_0)$ and
  $A(\lambda_1)$ are $\pm\imi/2$.
\end{proof}
\begin{remark}
  It is possible to compute the unitary factor $G$ of $\exp(\imi y A)$
  explicitly in terms of elliptic functions. It turns out that $G$ is
  quasiperiodic in $y$ (i.e., it is equivariant, and the period
  depends on $q$), and that the Iwasawa decomposition fails twice
  in each period. On the other hand, using~\eqref{eq:lightconematrix} it is
  possible to extend the Delaunay surface to a conformal immersion of
  the cylinder $\bbC^\ast$ into $\spaceS$; see
  also~\cite[$\S6$]{Babich_Bobenko_1993} and figure~\ref{fig:delaunay-h3}.
\end{remark}
\begin{remark}
  It is worth noting that the intersection of a Delaunay cylinder with
  the boundary at infinity is the disjoint union of circles.  This
  follows from the fact that Delaunay cylinders are equivariant.  If
  we restrict to one component of the Delaunay cylinder inside
  $\spaceH$ we obtain exactly two boundary circles, which define a
  Riemann surface of annulus type. It would be interesting to work out
  in detail the relation between the free parameter $q$ and the
  modulus of the annulus.
\end{remark}

\subsection{$n$-noids}

An \emph{$n$-noid} in $\bbR^3$ is a minimal immersion of a $n$
punctured Riemann sphere, each of whose end monodromies has Delaunay
eigenvalues, that is, the same eigenvalues as those of a Delaunay
cylinder.

In the previous example we have seen that a Delaunay end in $\spaceH$
cannot be defined on a punctured disk when we consider the surface
lying only in $\spaceH$. We therefore have to modify our definition:
\begin{definition}
  \label{def:nnoid}
  An \emph{$n$-noid} in $\spaceH$ is a conformal immersion
  \begin{equation}
    f\colon \bbCP^1\setminus\{p_1, \dots, ,p_n\}\to \mathcal L\cong \spaceS
  \end{equation}
  such that
  \begin{enumerate}
  \item the intersection \begin{equation}
    \text{image}(f)\setminus(\mathcal L\cap P (e_4)^\perp)) =
    \text{image}(f)\setminus S^2=\text{image}(f)\cap(\spaceH\cup\spaceH)
  \end{equation}
    is a (not necessarily connected) minimal surface;
  \item the surface has Delaunay eigenvalues around each end $p_k$.
  \end{enumerate}
\end{definition}
\begin{remark}
  It is necessary to explain the second condition in more detail. In
  general, if the surface passes through the boundary at infinity, the
  associated family of flat connections (remark~\ref{rem:assofami})
  does not exist on the $n$-punctured sphere
  $\bbCP^1\setminus\{p_1, \dots, p_n\}$ and it is therefore not obvious in
  which sense one should test the second condition. For example, one
  should expect that the intersection of
  $\text{image}(f)\cap(\spaceH\cup\spaceH)$ around an end $p_k$ is
  defined on a nested union of disjoint topological annuli, and
  a priori it is unclear why the eigenvalues of the monodromy on all
  annuli are the same.  On the other hand, using condition (1), the
  surface $f$ is a Willmore surface in $\spaceS$ and has an associated
  family of flat $\matSL(4,\bbC)$-connections which reduces (in a
  $\lambda$-dependent way) to the associated family of rank 2
  connections of the minimal surface in $\spaceH$ on the corresponding
  subset.  In this way, it can be shown that the monodromy
  representation up to conjugation is well-defined;
  for details see~\cite{Heller_Heller_Ndiaye_2018}.

  On the other hand, for minimal surfaces constructed from \DPW
  potentials the monodromies of the potential and the associated
  family agree up to conjugation, and the eigenvalue condition can
  therefore be checked directly on the potential.
\end{remark}
\begin{example}
  All minimal Delaunay cylinders are 2-noids. In fact, it follows
  from~\cite{Babich_Bobenko_1993} that these surfaces can be extended
  through the boundary at infinity to give a (Moebius-)periodic
  surface into $\spaceS$ from the two-punctured sphere.
\end{example}
So far we only have numerical evidence of the existence of $n$-noids
in $\spaceH$ (figures~\ref{fig:noid-h3},~\ref{fig:noid-h3-iso}
and~\ref{fig:trinoid-h3-halfspace}). We are therefore forced
to give the following weaker definition.
\begin{definition}
  \label{def:openn}
  An \emph{open $n$-noid} in $\spaceH$ is a conformal minimal immersion of a
  $n$-holed sphere
  \begin{equation}
    f\colon \bbCP^1\setminus (D_1\cup \cdots \cup D_n)\to\spaceH
  \end{equation}
  for non-intersecting topological disks $D_k\subset\bbCP^1$, such
  that the monodromy eigenvalues around each hole $D_k$
  are the monodromy eigenvalues of a
  Delaunay cylinder for some $q\in\bbR$ with $\abs{q}>1$.
\end{definition}
\begin{remark}
  Not every $n$-noid in $\spaceH$ is an open $n$-noid in the above
  sense. For example, it might be the case that two or more {\em ends}
  of the $n$-noid start in each of the two $\spaceH$ copies inside the
  conformal 3-sphere, while the two pieces are joined by a surface
  which is close to a part of a round sphere. In fact, such examples
  can be constructed by the methods below. On the other hand, we will
  construct open $n$-noids in theorem~\ref{thm:main}. We conjecture that
  those surfaces are $n$-noids in the sense of definition~\ref{def:nnoid}
  (figures~\ref{fig:noid-h3},~\ref{fig:noid-h3-iso}
  and~\ref{fig:trinoid-h3-halfspace}).
\end{remark}

\subsubsection{3-noids}


A potential for trinoids is
\begin{equation}
  \label{eq:trinoid}
  \begin{bmatrix}0 & \lambda^{-1}\\ f(\lambda) Q & 0\end{bmatrix}\deriv z,
\end{equation}
where $Q\deriv z^2$ is a holomorphic quadratic differential with three
double poles and real quadratic residues, $\lambda_0,\,\lambda_0^{-1}$
are the \sym points, and for the ambient space $\spaceH$
\begin{align}
  \quad
  f = (\lambda -1)(\lambda + 1).\quad
\end{align}
Since at each of its poles the potential is gauge equivalent to a
perturbation of a Delaunay potential, by the theory of regular
singular points, each monodromy around a puncture has Delaunay
eigenvalues. This potential constructs CMC trinoids if the closing
conditions of remark~\ref{rem:closing} are satisfied,
as shown by the following theorem.
\begin{theorem}
  \label{thm:trinoidH}
  There exists a real 1-parameter family of \DPW potentials on the
  3-punctured sphere satisfying the intrinsic and extrinsic closing
  conditions for minimal surfaces in $\spaceH$.
\end{theorem}
The technical proof of the theorem is given in section~\ref{sec:g3noid} below.

\subsection{Existence of $n$-noid potentials with small necksize}

\begin{figure}[b]
  \centering
    \includegraphics[width=0.49\textwidth]{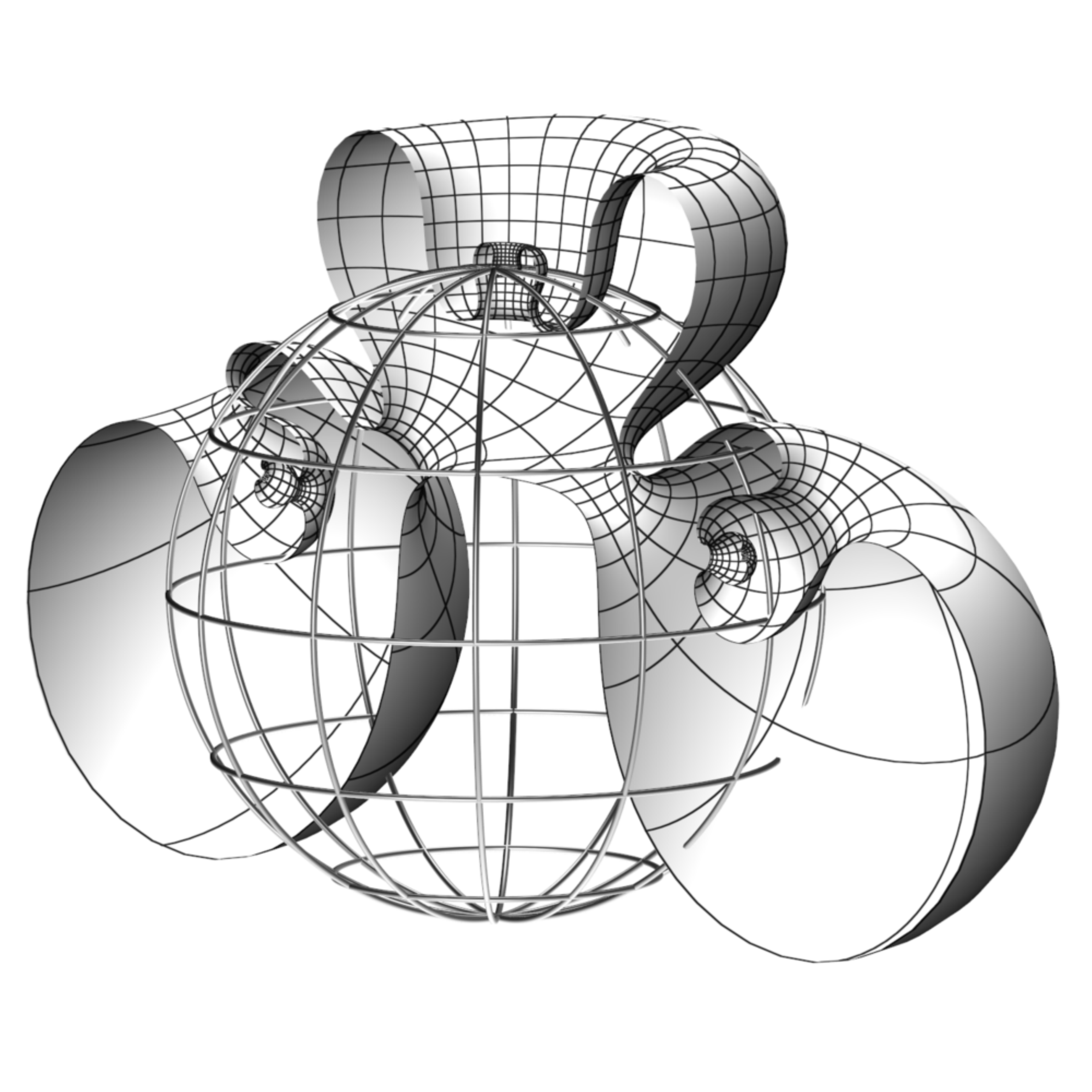}
    \includegraphics[width=0.49\textwidth]{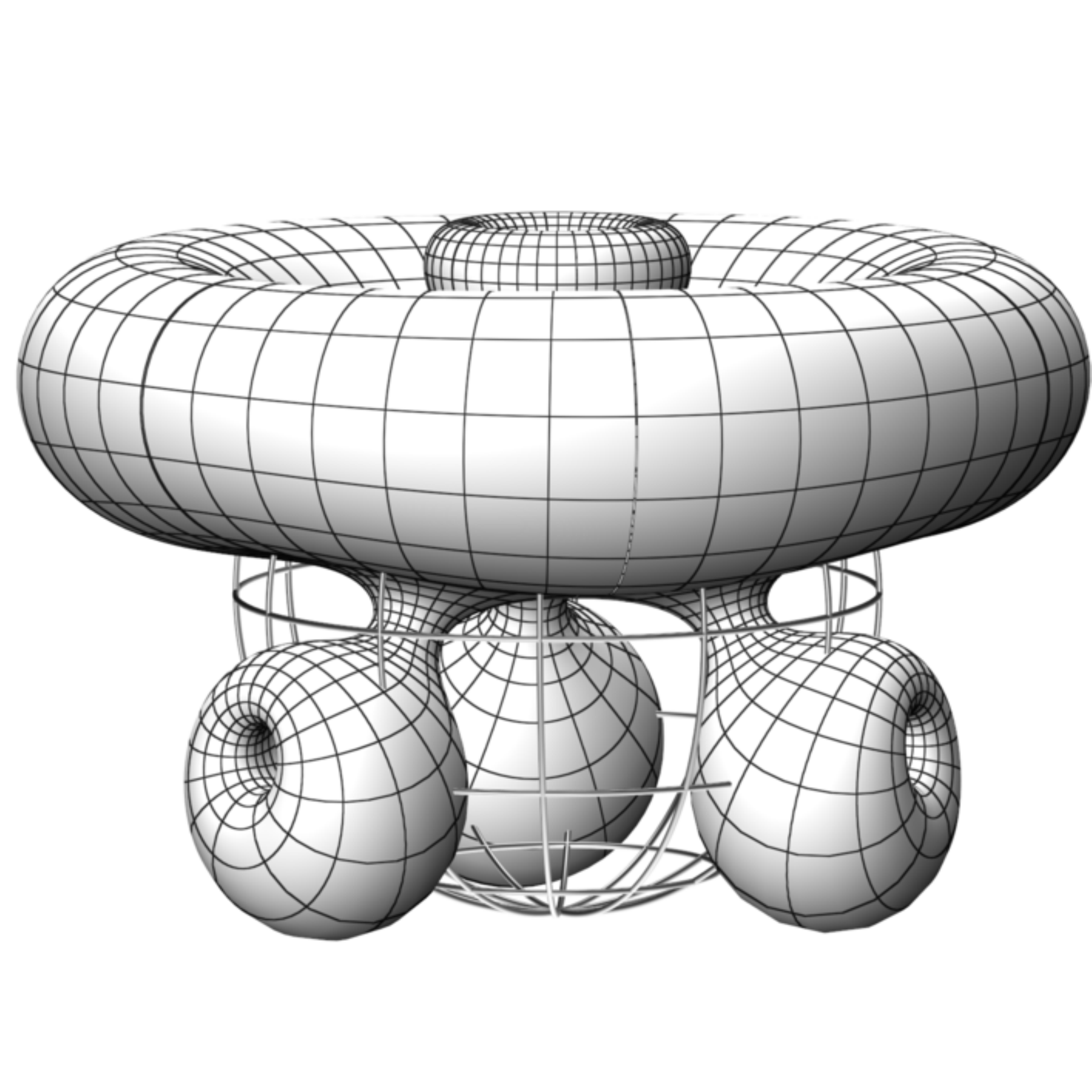}
  \caption{\small
    Trinoid and fournoid in $\spaceH\cup\Stwo\cup\spaceH$
    with respective cyclic symmetries of orders $2$ and $3$
    about the vertical axis. Neither surface is equilateral.}
  \label{fig:noid-h3-iso}
\end{figure}

We adopt the techniques of Traizet~\cite{Traizet_2017} to prove the
existence of open $n$-noids in $\spaceH$.  We conjecture that these
surfaces are $n$-noids in the sense of definition~\ref{def:openn} as well.

In~\cite{Traizet_2017}, Traizet showed the existence of \DPW
potentials for constant mean curvature $n$-noids in $\bbR^3$ by
deforming the \DPW potential of the round sphere. His method of solving
the monodromy problem for the intrinsic and extrinsic closing
conditions can be easily translated to our setup, with basically
identical proofs up to minor changes. Additionally, one can deduce
that the Iwasawa factorization works on a subset homeomorphic to a
$n$-holed sphere, which yields actually examples of open $n$-noids in
$\spaceH$.  A formally similar method of deforming surfaces in
$\spaceS$ and $\spaceH$ has been introduced
in~\cite{Heller_Heller_Schmitt_2018} and~\cite{Heller_Heller_2018}
respectively.

We start with a potential
\begin{equation}
  \label{eq:H3potential}
  \xi_t(z):=
  \begin{bmatrix} 0 &\lambda^{-1} \deriv z\\ t (\lambda^2+1) \omega(z)\end{bmatrix}
\end{equation}
where
\begin{equation}
  \omega=
  \sum_{k=1}^n \bigg(\frac{a_k}{(z-z_k)^2}+\frac{b_k}{z-z_k} \bigg)\deriv z.
\end{equation}
We call $\mathbf{x}=(a_1,\dots,z_n)$ the parameter of the potential.
Note that in~\eqref{eq:H3potential} we take the \sym points to be
$\lambda_0=\imi$ and $\lambda_1=-\imi$ in order to have more natural
reality conditions.  As in~\cite{Traizet_2017}, we need to allow that
the coefficients $a_k,\, b_k,\, z_k$ are holomorphic functions in
$\lambda$, i.e., they are holomorphic on an open neighborhood of the
closed unit disk in the $\lambda$-plane. They need to be adjusted for
small $t>0$ such that the intrinsic closing condition is satisfied: we
want to find for small $t>0$ holomorphic functions $a_k,\, b_k,\,
z_k$, which are close to constant functions (satisfying a constraint
related to some balancing formula) , such that the monodromy based at
$z=0$ of the potential~\eqref{eq:H3potential}
is in the unitary loop group determined
by~\eqref{eq:real-form} corresponding to the symmetric space
$\spaceH$. As we suppose that the functions $z_k$ are close enough to
constants, the potential is well-defined on a $n$-holed sphere for all
$\lambda\in \{\lambda\in\bbC^\ast\mid 0<\abs{\lambda}<1+\epsilon\}$
for some $\epsilon>0$, and the monodromy is computed on this $n$-holed
sphere. We call $\mathbf{x}=(t,a_1, \dots, z_1, \dots)$ the parameter of the
potential, even in the case when $a_1, \dots, z_1,\dots$ depend on $\lambda$.

For $\epsilon>0$, we denote by $\mathcal B^\epsilon$ the Banach space
of holomorphic functions on
\begin{equation}
  \{\lambda\in\bbC\mid \abs{\lambda}<1+\epsilon\},
\end{equation}
equipped with the generalized Wiener norm as in~\cite[$\S$4]{Traizet_2017}.
\begin{lemma}\label{lem:potential}
  For $k=1, \dots, n$ let $\tau_k\in\bbR\setminus\{0\}$ and $p_k\in
  \bbC\setminus(\{0\}\cup \Sone$ such that
  \begin{equation}
    \label{eq:weight}
    0=\sum_{k=1}^n
    \frac{2\tau_k \overline{p_k}}{1-\abs{p_k}^2}=\sum_{k=1}^n\tau_k
    \frac{1+\abs{p_k}^2}{1-\abs{p_k}^2}=\sum_{k=1}^n\frac{2\tau_k p_k}{1-\abs{p_k}^2}.
  \end{equation}
  Then there exists $\epsilon$, $T>0$ and unique smooth maps
  \begin{equation}
    b_k, z_k\colon [0;T[\ \to\ \mathcal B^\epsilon,\quad k = 1, \dots, n
  \end{equation}
  with \begin{equation}
    b_k(0)=\frac{2\tau_k \overline{p_k}}{1-\abs{p_k}^2},\quad z_k(0)=p_k
  \end{equation}
  for $k=1, \dots, n$, and $z_k(0)(0)=p_k$ for $k=1, \dots, n$, and a smooth functions
  $\tau_k\colon[0,T[\to\bbR$ with $\tau_k(0)=\tau_k,$ $k=1, \dots, n$ such that for
      $t\in[0;T[$ the potential~\eqref{eq:H3potential} with parameter
          \begin{equation}
            \mathbf{x}=(t,\tau_1, \dots, \tau_{n},a_n(t),b_1(t), \dots, z_n(t))
          \end{equation}
          has $\spaceH$-unitary monodromy at $z=0$.
\end{lemma}
\begin{proof}
  The proof is analogous to the proof of~\cite[Proposition 3]{Traizet_2017},
  under the extra assumption that $\abs{p_k}\neq1$
  for $k=1, \dots, n$, and using the adapted $\ast$-operator on functions,
  i.e.\ $f^\ast(\lambda)=\overline{f(-1/\bar\lambda)}$. The proof that the functions $\tau_k$ are independent
 of $\lambda$ and real-valued can be done similarly to the proof of~\cite[Proposition 4]{Traizet_2017} by  looking at  the eigenvalues of the local monodromies.
\end{proof}
As a corollary we obtain our main theorem:
\begin{theorem}
  \label{thm:main}
  For every $n>1$ there exist open minimal $n$-noids in $\spaceH$.
\end{theorem}
\begin{proof}
  For $n=2$ we have seen the existence of Delaunay cylinders. For
  $n\neq3$ it is easy to see the existence of $n$ pairwise distinct
  points $p_k\in\bbC^\ast$ with $0<\abs{b_l}<1$ for $l=1, \dots, n-1$ and
  $\abs{p_n}>1$ and $n$ positive real numbers $\tau_k$ satisfying the
  balancing formula~\eqref{eq:weight}.

  To construct an open minimal $n$-noid we consider for $T>t>0$ the
  potential provided by Lemma~\ref{lem:potential}, and the solution
  $\Phi$ of
  \begin{equation}
    d\Phi=\Phi\xi_{t};\quad \Phi(0)=\id.
  \end{equation}
  The Iwasawa decomposition (for the real involution corresponding to
  $\spaceH$) does exist on an open subset of the loop group
  $\Lambda\matSL_2(\bbC)$, containing the constant loop $\id$.
  As for $\delta>0$ small enough and $t$ small enough $\xi_{t}$ is
  arbitrarily close to $\xi_{0}$ on
  \begin{equation}
    \Sigma=\{z\in\bbC\mid \abs{z}<1-\delta; \abs{z-p_k}>\delta
    \text{ for } k=1, \dots, n-1\}
  \end{equation}
  the loop $\Phi(z)$ admits an Iwasawa decomposition on the $n$-holed
  sphere $\Sigma$, and theorem~\ref{thm:gwr} for the \sym points
  $\lambda_0=\imi$, $\lambda_1=-\imi$ provides a minimal surface
  \begin{equation}
    f\colon\Sigma\to\spaceH.
  \end{equation}
  By construction (and for $\delta$ small enough), the monodromies
  around the holes have Delaunay eigenvalues since at each of its
  poles the potential is gauge equivalent to a perturbation of a
  Delaunay potential.
\end{proof}
Figure \ref{fig:16noid-h3} presents an equilateral minimal 16-noid. This is a surfaces with dihedral symmetry. Our numerical experiments show that such surfaces stay embedded for an arbitrary number of ends.
\begin{figure}[b]
  \centering
  \includegraphics[width=0.75\textwidth]{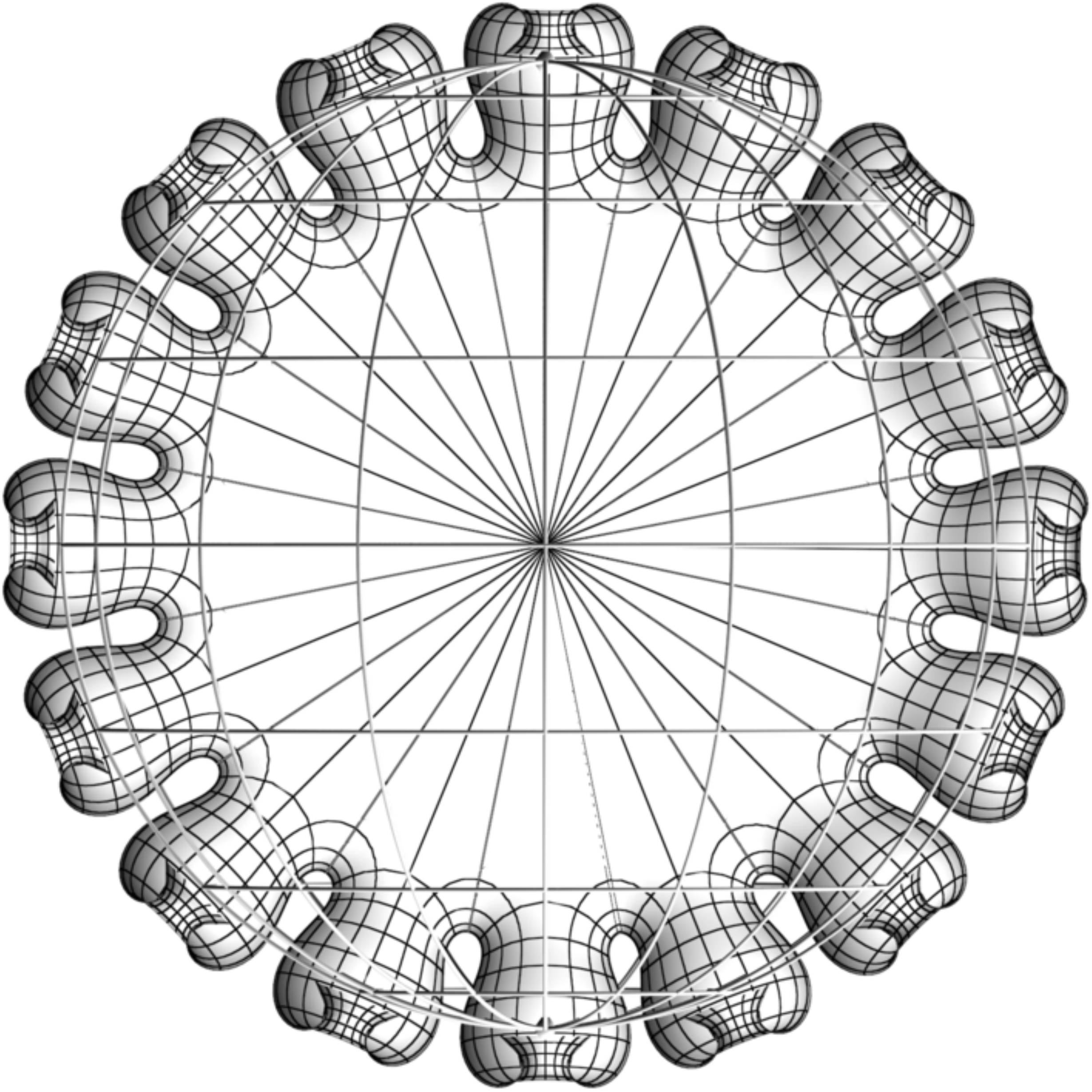}
  \caption{\small
    An equilateral minimal 16-noids in $\spaceH\cup\Stwo\cup\spaceH$
    with cyclic symmetry of order $16$.}
  \label{fig:16noid-h3}
\end{figure}
\begin{conjecture}
There exist embedded minimal $n$-noids in $H^3$ for any $n\in\mathbb N^{\geq2}$.
\end{conjecture}

\section{Minimal $n$-noids in anti-de~Sitter space $\spaceADS$}
\label{sec:ads3}

\subsection{The lightcone model for $\spaceADS$}
We set
\begin{equation}\spaceADS= \matSU_{11}\end{equation}
equipped with the natural biinvariant Lorentzian metric associated to
the quadratic form given by the trace. As for hyperbolic space,
anti-de~Sitter space can be defined as the complement of the boundary
at infinity
\begin{equation}
  S_\infty=\{\bbR (x_0,\,x_1,\,x_2,\,x_3,\,0)\mid x_0^2+x_1^2-x_2^2-x_3^2=0\}
\end{equation}
of the lightcone
\begin{equation}
  \mathcal L=\{\bbR (x_0,\,x_1,\,x_2,\,x_3,\,x_4)\mid x_0^2+x_1^2-x_2^2-x_3^2-x_4^2=0\}
\end{equation}
via
\begin{equation}
  \bbR(x_0,\,x_1,\,x_2,\,x_3,\,x_4)\in\mathcal L\setminus S_\infty\mapsto
  \frac{1}{x_4}
  \begin{bmatrix}x_0+i x_1& x_2-i x_3\\ x_2+ ix_3& x_0-i x_1\end{bmatrix}
    \in \matSU_{11}.
\end{equation}
For visualization, we make use of the {\em stereographic projection}
\begin{equation}
  [x_0,\,x_1,\,x_2,\,x_3,\,x_4]\in\mathcal L^0\mapsto
  \frac{1}{x_0+x_4} (x_1,\,x_2 ,\,x_3)\in\bbR^{1,2}
\end{equation}
defined on
\begin{equation}
  \mathcal L^0=\mathcal L\setminus
  \{ \bbR(x_0,\,x_1,\,x_2,\,x_3,\,-x_0)\mid x_1^2-x_2^2-x_3^2=0\}.
\end{equation}
This map is a conformal diffeomorphism onto an open subset of
Minkowski space $\bbR^{1,2}$.

\subsection{Basic examples}
We first describe some simple surfaces in terms of the \DPW approach.
\subsubsection{The sphere in $\spaceADS$}
As in the case of $\spaceS$ and $\spaceH$
the easiest example of a \DPW potential is given by
\begin{equation}
  \xi(\lambda)=\begin{bmatrix}0 & \lambda^{-1} \\
  0 & 0\end{bmatrix}dz
\end{equation}
on the complex plane, with \DPW frame
\begin{equation}
  \Phi(\lambda)=\begin{bmatrix}1 & \lambda^{-1} z \\
  0 & 1\end{bmatrix}.
\end{equation}
The factorization for $\spaceADS$ is given by
\begin{equation}
  \Phi(\lambda)=F(\lambda)B(\lambda)=\frac{1}{\sqrt{1-z\bar z}}
  \begin{bmatrix}1 & \lambda^{-1} z \\
    \lambda \bar z & 1\end{bmatrix}\frac{1}{\sqrt{1-z\bar z}}\begin{bmatrix}1 & 0 \\
    -\lambda \bar z & 1-z\bar z\end{bmatrix}
\end{equation}
and taking \sym points $\lambda_0=1$ and $\lambda_{1}=-1$ we obtain
\begin{equation}
  \label{eq:ads3sphere}
  f=F(1)F(-1)^{-1}=\frac{1}{1-z\bar z}
  \begin{bmatrix} 1+ z\bar z& 2z \\ 2\bar z &1+ z\bar z\end{bmatrix}.
\end{equation}
Restricted to the unit disk $D\subset\bbC$ this is just a
conformally parametrized totally geodesic hyperbolic disk inside
$\matSU_{11}=\spaceADS$ with induced metric
\begin{equation}
  \frac{1}{1-z\bar z}\deriv z\otimes \deriv\bar z.
\end{equation}
\begin{remark}
The surface $f$ is not well-defined on the whole plane $\bbC$ or
projective line, but crosses the ideal boundary at infinity $S_\infty$
along the unit circle $\Sone\subset\bbC$. By~\eqref{eq:ads3sphere} $f$
can be continued on the complement of the closed unit disk as a map
into $\matSU_{11}$. Again, this phenomena turns out to be typical in
the examples below
(figures~\ref{fig:delaunay-ads3} and~\ref{fig:trinoid}).
\end{remark}

\subsubsection{Minimal planes}

Besides the trivial example of the hyperbolic disk, there is an
interesting class of minimal surfaces with trivial topology in
$\spaceADS$ which have been investigated in detail in~\cite{Alday_Maldacena_2009},
and are parametrized by null polygonal
boundaries. In the special case of regular polygons the Hopf
differential is given by a homogeneous polynomial on the complex
plane. The metric is rotationally invariant and given by a global
solution of the Painlev\'e III equation.

Analogous to the \DPW description \cite{Bobenko_Its_1995} of \emph{Smyth surfaces} in $\bbR^3$ with rotationally
symmetric metric~\cite{Smyth_1993} these minimal planes can be
constructed via the \DPW potential of the form
\begin{equation}
  \label{eq:smyth}
  \begin{bmatrix}
    0 & \lambda^{-1} \\ cz^n & 0\end{bmatrix}\deriv z,\quad
    n\in\bbN,\quad c\in\bbR
\end{equation}
for appropriate $c \in\bbR\setminus\{0\}$. 
 The surface (with
initial condition at $z=0$ given by $\id$) has discrete ambient
cyclic symmetry of order $n+2$.  The Iwasawa decomposition is global
only for one special value $\hat c$ (depending on $n$); see for
example the detailed investigations in~\cite{Guest_Its_Lin_2018}. In
figure~\ref{fig:smyth-ads3} a surface for $n=1$ with $c$ numerically
close to $\hat c$ is shown. A systematic
investigation of this boundary behavior of solutions of the underlying
Gauss equation, i.e., the generalized sinh-Gordon equation, on
punctured Riemann surfaces with Hopf differentials having higher order
poles is carried out in~\cite{Gupta_2018},
under the name {\em crowned Riemann surfaces}.

\subsubsection{Delaunay cylinders in $\spaceADS$}

\begin{figure}[b]
  \centering
  \begin{subfigure}[t]{0.49\textwidth}
    \centering
    \includegraphics[width=\textwidth]{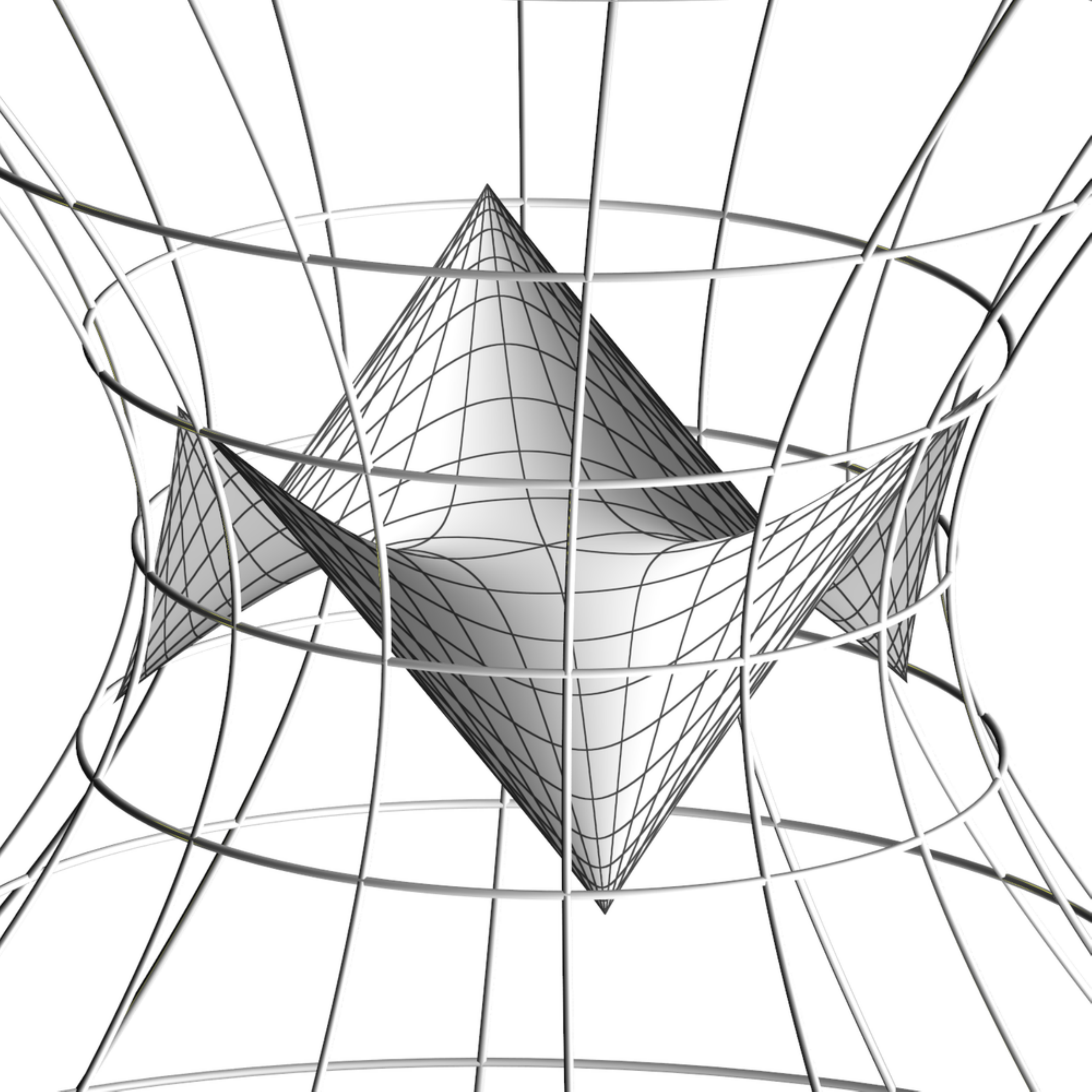}
    \caption{
      \small
      Minimal Smyth surface in $\spaceADS$ with cyclic symmetry of order three.}
    \label{fig:smyth-ads3}
  \end{subfigure}
  \begin{subfigure}[t]{0.49\textwidth}
    \centering
    \includegraphics[width=\textwidth]{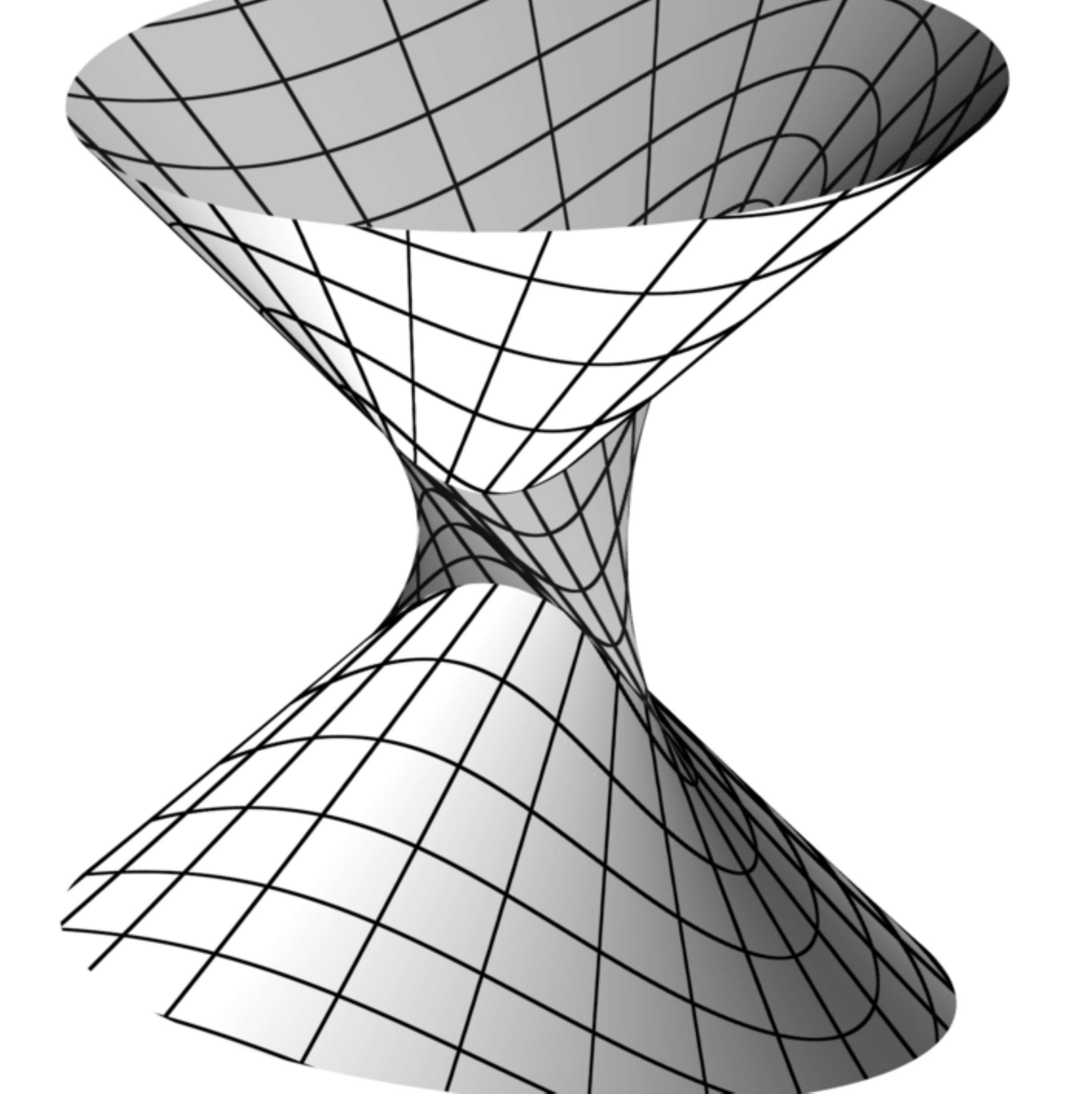}
    \caption{\small
      Swallowtail singularity on a minimal surface in $\spaceADS$.
    }
    \label{fig:swallowtail-ads3}
  \end{subfigure}
  \caption{}
  \label{fig:smyth}
\end{figure}

Minimal Delaunay planes in $\spaceADS$ have been described in
\cite{Bakas_Pastras_2016} in terms of their elliptic spectral data. By
imposing the extrinsic closing condition along the equivariant
direction they form a real 1-dimensional family of geometrically
distinct surfaces.  The conformal factor with respect to the
coordinate $\deriv w$ on $\bbC/(2\pi \imi\bbZ)$ is a solution of the
sinh-Gordon equation with opposite sign as for CMC surfaces in
$\spaceS$~\eqref{eq:gauss-equation}, and the Hopf differential
is a constant multiple of $(\deriv w)^2$. The conformal factor
can be given explicitly in terms of the Weierstrass $\wp$-function
on a rectangular elliptic curve; see~\cite{Bakas_Pastras_2016}. The surface is
rotationally symmetric, and the conformal factor depends only on one
(real) variable and is periodic, but blows up once each (intrinsic)
period where the surface intersects the ideal boundary at infinity.
Moreover, the conformal factor $v^2=e^{2u}$ vanishes to second order
once in each (intrinsic) period.  Geometrically this means that one
tangent direction is mapped to a lightlike vector, while the
orthogonal tangent direction (with respect to the Riemann surface
structure) spans the kernel of the differential of the CMC surface.
In the equivariant case, at the singular points, the kernel of the
differential of the CMC surface coincides with the kernel of the
differential of the conformal factor $v^2$, giving us a cone point
(figure~\ref{fig:delaunay-h3}).  This phenomena does not happen
generally, where a swallowtail like
singularity is expected (figure~\ref{fig:swallowtail-ads3}).

In the following, we consider the Delaunay cylinders parametrized on
the 2-punctured sphere $\bbC^\ast$, where the Hopf differential is
a constant multiple of $(\deriv z)^2/z^2$. These can be
constructed via theorem~\ref{thm:gwr} analogously to Delaunay
cylinders in $\spaceH$.

A family of \DPW potentials inducing Delaunay surfaces on $\bbC^\ast$
with initial condition $\Phi(1)=\id$ is
\begin{equation}
  \label{eq:delaunay}
  A\frac{\imi\deriv z}{z},\quad
  A = B + B^\ast,\quad
  B = \begin{bmatrix}\tfrac{\imi c}{2} & a\lambda^{-1}\\
    b & -\tfrac{\imi c}{2}\end{bmatrix},\quad
  a\in\bbR^\ast,\quad b,\,c\in\bbR.
\end{equation}
Imposing the extrinsic closing condition with $\lambda_0=\imi$,
a $1$-parameter family of potentials is given by
\begin{equation}
  B = \frac{1}{2\sqrt{q^2+1}}
  \begin{bmatrix}\imi \sqrt{2(q^2+1)} & \lambda^{-1} + q\\
    \lambda + q & -\imi \sqrt{2(q^2+1)}\end{bmatrix}\\
\end{equation}
for $q\in\bbR$.
The meromorphic frame $\Phi$ is based at $z=1$ with $\Phi(1) = \id$.
As in the previous section we obtain:
\begin{theorem}
  The \DPW construction
  applied to the above data
  gives Delaunay cylinders in $\spaceADS$.
\end{theorem}
An example is shown in figure~\ref{fig:delaunay-ads3}.

\subsection{$n$-noids in $\spaceADS$}

We may define $n$-noids and open $n$-noids in $\spaceADS$ in an
analogous way as for surfaces in $\spaceH$. We can also adopt
Traizet's construction~\cite{Traizet_2017} as done for the case of
minimal surfaces in $\spaceH$ above to prove the existence
of open $n$-noids in $\spaceADS$. It would be nice to prove that these
examples are actually $n$-noids in the strong sense; a proof of this
conjecture might build up on the techniques developed in
\cite{Raujouan_2018}.

\begin{theorem}
For every $n>1$ there exist open minimal $n$-noids in $\spaceADS$.
\end{theorem}

The proof is analogous to the proof of theorem~\ref{thm:main},
taking the appropriate $\ast$-operator on
$\matSL_2(\bbC)$ and on holomorphic functions. Rather than repeating
the details of the arguments we give in the next section
a general proof of existence
of 3-noid potentials for surfaces in $\spaceS$, $\spaceH$, $\spaceADS$ and
$\spaceDS$.

\section{$3$-noids in symmetric spaces}
\label{sec:g3noid}

In this last section we extend slightly the class of surfaces and
consider CMC surfaces in the symmetric spaces $\spaceS$, $\spaceH$,
$\spaceADS$ and $\spaceDS$. Recall from section~\ref{sec:CMCDPW} that these
can also be described by the \DPW approach.  We present a general
proof of existence for \DPW potentials on a 3-punctured sphere,
denoted as trinoids in the following.

The main motivation for giving an explicit proof of existence for
trinoids is that the so constructed surfaces can be numerically
computed and visualized directly
(figures~\ref{fig:noid-h3},~\ref{fig:noid-h3-iso},~\ref{fig:trinoid-h3-halfspace} and~\ref{fig:trinoid}).
In particular, the
proof might help the reader to perform computer experiments via \DPW
method. Also, in the case of a 3-punctured sphere, the
range of end weights $q$ (determined by the quadratic residues at the
ends) is more tractable here than in the implicit function
theorem proof of theorem~\ref{thm:main}.

The construction of trinoids  is particularly simple
because the conjugacy class of the monodromy group in this case is
determined by the individual conjugacy classes of the
generators. In the case of the involution $\lambda\mapsto  1/\ol{\lambda}$
the closing conditions can be read off from the
monodromy eigenvalues, which by the theory of regular singular points,
can be read off from the potential. 

Each pole of our potential will be a Fuchsian singularity
which is \DPW gauge equivalent to a
perturbation of a Delaunay potential.
By the theory of regular
singular points, each monodromy around a puncture has Delaunay
eigenvalues. This potential constructs CMC trinoids if the closing
conditions of remark~\ref{rem:closing} are satisfied.  The following
theorem constructs trinoids in
$\spaceADS$, $\spaceH$ with $\abs{H}<1$, and $\spaceDS$ with $\abs{H}<1$,
in analogy to the construction of trinoids in $\spaceS$
in~\cite{Schmitt_Kilian_Kobayashi_Rossman_2007}.

To show that the trinoid closes it is necessary to compute a
unitarizer $X$ of the monodromy.  Unlike the case of $\spaceS$, for
the other spaceforms the unitarizer could fail to have an Iwasawa
factorization.  Then the frame with unitary monodromy $X\Phi$ likewise
fails to have an Iwasawa factorization, and hence does not construct a
trinoid.  Hence it is necessary to find a unitarizer which has an
Iwasawa factorization, or equivalently, a unitarizer in
$\Lambda_+\matSL_2(\bbC)$.

To solve the technical problem of the existence of a monodromy
unitarizer in $\Lambda_+\matSL_2(\bbC)$ we make the simplifying
restriction to \emph{isosceles} trinoids, for which two of the three
end parameters are equal. As seen in theorem~\ref{thm:trinoid}, under this
restriction, a diagonal monodromy unitarizer can be computed
explicitly via a scalar Birkhoff factorization.

\subsection{Real forms}

\begin{figure}[b]
  \centering
  \begin{subfigure}[t]{0.49\textwidth}
    \includegraphics[width=\textwidth]{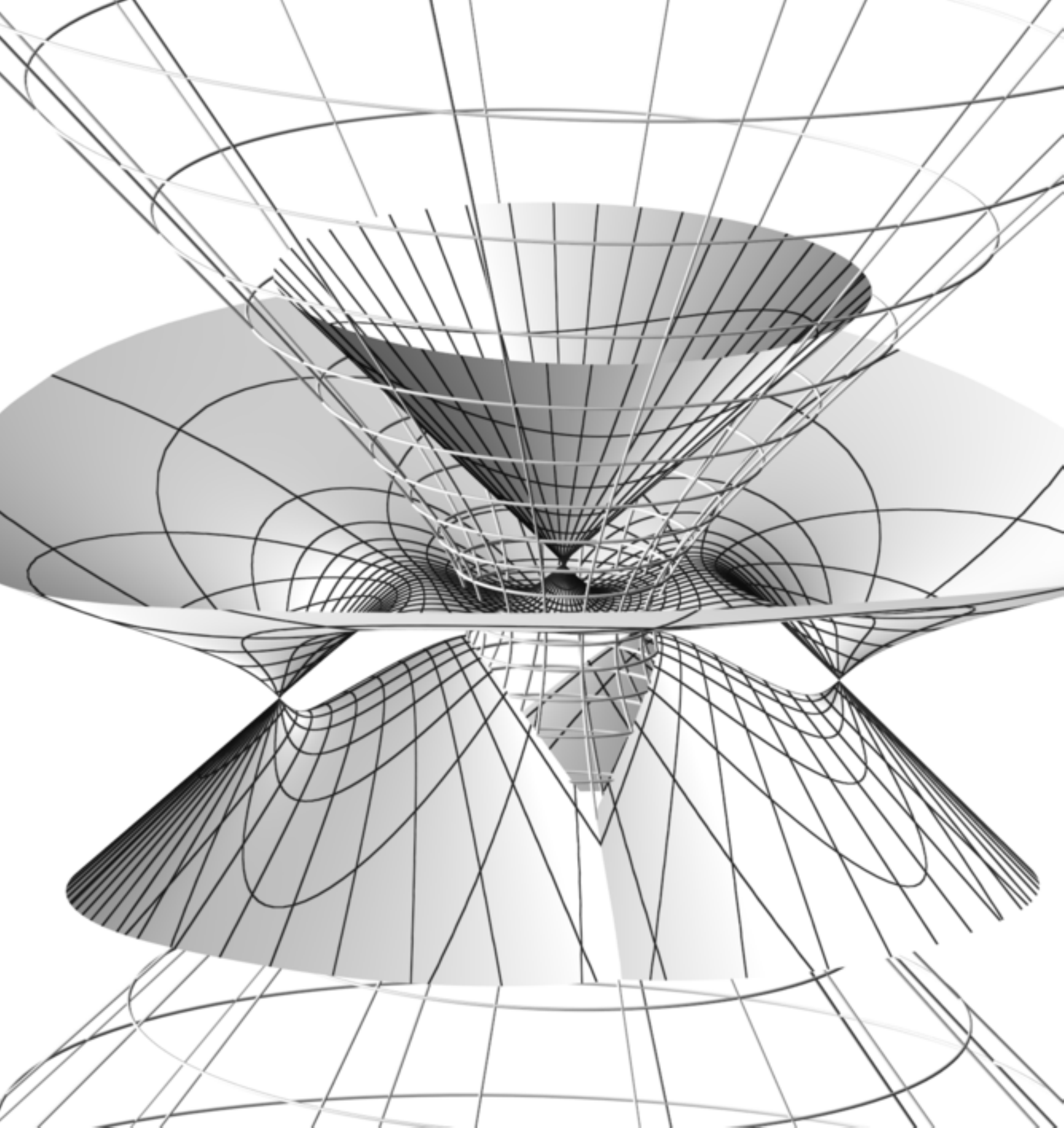}
    \caption{\small
      Minimal trinoid in $\spaceADS$
    }
    \label{fig:trinoid-ads3}
  \end{subfigure}
  \begin{subfigure}[t]{0.49\textwidth}
    \includegraphics[width=\textwidth]{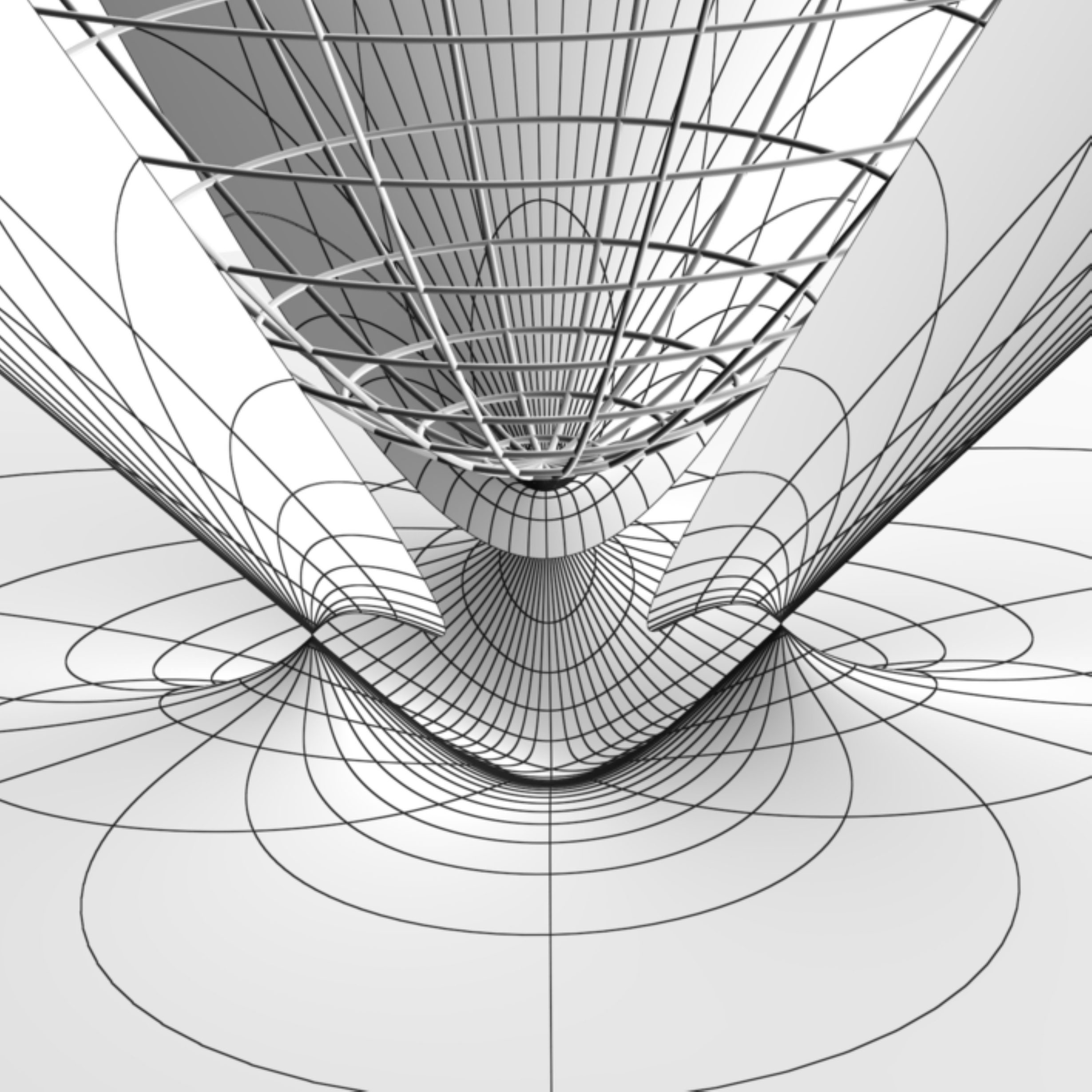}
    \caption{\small
      Minimal trinoid in $\spaceDS$
    }
    \label{fig:trinoid-ds3}
  \end{subfigure}
  \caption{\small
    Minimal trinoids in $\spaceADS$ and $\spaceDS$:
    three Delaunay half-cylinders glued
    to a minimal two-sphere (horizontal plane).
    A swallowtail singularity (figure~\ref{fig:swallowtail-ads3})
    appears on each Delaunay end.
  }
  \label{fig:trinoid}
\end{figure}

The four involutions~\eqref{eq:real-form}
on loops $X:\Sone\to\matSL_2\bbC$
can be indexed by $\delta\in\{\pm 1\}$ and $\epsilon\in\{\pm 1\}$:
\begin{equation}
  \label{eq:unitary}
  X^\ast(\lambda) = {\transpose{\ol{ \eta X(\delta/\ol{\lambda}) \eta^{-1} }}}^{-1},\quad
  \eta = \begin{cases}
    \id &
    \text{if $\epsilon=1$}\\
    \diag(\imi,\,-\imi) &
    \text{if $\epsilon=-1$}.
  \end{cases}
\end{equation}
The subgroup of loops $X$ satisfying $X^\ast = X$
is the real form for the symmetric spaces as tabulated:
\begin{equation}
  \label{eq:trinoid-table}
  \begin{array}{c|c|c}
    &\epsilon=1 & \epsilon=-1\\
    \hline
    \delta=1 & \spaceS & \spaceADS\\
    \delta=-1 & \spaceH\ (\abs{H}<1) & \spaceDS\ (\abs{H}<1)
  \end{array}
\end{equation}
A loop $M:\Sone\to\matSL_2\bbC$ is \emph{unitary} if $M^\ast = M$.
A monodromy group is \emph{unitarizable}
if there exists a map $X:\disk_+\to\matSL_2\bbC$
such that for each $M$ in the group $XMX^{-1}$ extends holomorphically to $\Sone$ and is unitary.

For scalar loops define
\begin{equation}
  \label{eq:scalar-star}
  f^\ast(\lambda) = \ol{ f(\delta/\ol{\lambda}) }.
\end{equation}

\subsection{Trinoids potentials}

A potential for trinoids is
\begin{equation}
  \label{eq:trinoid-potential}
o  \begin{bmatrix}0 & \lambda^{-1}\\ f(\lambda) Q & 0\end{bmatrix}\deriv z,
\end{equation}
where $Q\,\deriv z^2$ is a holomorphic quadratic differential with three
double poles and real quadratic residues,
$f$ satisfying $(f/\lambda)^\ast = f/\lambda$ is
\begin{subequations}
  \label{eq:trinoid-f}
  \begin{align}
  \label{eq:trinoid-f1}
    \spaceS\text{ and }\spaceADS&:\quad
    f = (\lambda - \lambda_0)(\lambda - \lambda_0^{-1}),\quad
    \lambda_0\in\Sone\setminus\{\pm 1\}\\
  \label{eq:trinoid-f2}
    \spaceH\text{ and }\spaceDS&:\quad
    f = (\lambda - \lambda_0)(\lambda + \lambda_0^{-1}),\quad
    \lambda_0\in\bbR\setminus\{0\},
  \end{align}
\end{subequations}
and
$\lambda_0,\,\lambda_0^{-1}$
are the \sym points  $\lambda_0,\lambda_1$ are determined by
\begin{equation}
  \label{eq:delaunay-sym-points}
  \spaceS \text{ and }\spaceADS:\ \lambda_1 = 1/\lambda_0\in\Sone\setminus\{\pm 1\};
  \quad\quad
  \spaceH\text{ and }\spaceDS:\ \lambda_1 = -1/\lambda_0\in\bbR^\ast.
\end{equation}

For isosceles trinoids choose ends $z=1,\,-1,\,\infty$ and
\begin{equation}
  \label{eq:trinoid-hopf}
  Q = \frac{4a + b(z^2-1)}{ {(z^2-1)}^2 }
\end{equation}
satisfying $\qres_{z= \pm 1}Q\,\deriv z^2 = a$ and $\qres_{z=\infty}Q\,\deriv z^2 = b$.

\subsection{Unitarization}

The isosceles trinoid potential~\eqref{eq:trinoid-potential}--\eqref{eq:trinoid-hopf}
has symmetries
\begin{subequations}
  \label{eq:trinoid-potential-symmetry}
  \begin{gather}
    \sigma^\ast\xi(\lambda) = \xi(\lambda) . g_1\quad
    \sigma(z)=-z,\quad
    g_1=\diag(\imi,\,-\imi)\\
    \ol{\tau^\ast\xi(\delta/\ol{\lambda})} = \xi(\lambda) . g_2,\quad
    \tau(z) =\ol{z},\quad
    g_2=\diag(\sqrt{\delta}/\lambda,\,\lambda/\sqrt{\delta}).
  \end{gather}
\end{subequations}
Let $M_0$ and $M_1$ be the monodromies around
$z=1$ and $z=-1$ respectively, with basepoint $\Phi(0)=\id$.
The symmetries of the potential imply
\begin{equation}
  \label{eq:trinoid-monodromy}
  M_0 = \begin{bmatrix}r & p\lambda\\ -q\lambda^{-1} & r^\ast\end{bmatrix},
    \quad
    M_1 = \begin{bmatrix}r & -p\lambda\\q\lambda^{-1} & r^\ast\end{bmatrix}
\end{equation}
for some holomorphic functions $p,\,q,\,r:\bbC^\ast\to\bbC$
satisfying $rr^\ast + p q = 1$ and
\begin{equation}
  p^\ast = \delta p\quad\text{and}\quad q^\ast = \delta q.
\end{equation}

\begin{lemma}
  \label{lem:trinoid-monodromy}
  If $q/p^\ast$ in~\eqref{eq:trinoid-monodromy} has a Birkhoff factorization
  \begin{equation}
    q/p^\ast = \delta\epsilon x_+^\ast x_+,\quad
    x_+:\disk_+\to\bbC^\ast
  \end{equation}
  then $M_0$ and $M_1$ are unitarizable in the sense of~\eqref{eq:unitary}.
\end{lemma}
\begin{proof}
  If $q/p^\ast = \delta\epsilon x_+^\ast x_+$, then
  $x_+:\disk_+\to\bbC^\ast$ has a single-valued square root $\disk_+\to\bbC^\ast$,
  so $X := \diag(\sqrt{x_+},\,1/\sqrt{x_+})$ is a single-valued
  map $\disk_+\to\matSL_2(\bbC)$. Then
  \begin{equation}
    P_0:=XM_0X^{-1} = \begin{bmatrix}r & s\\-\epsilon s^\ast & r^\ast\end{bmatrix},
      \quad
      P_1:=XM_1X^{-1} = \begin{bmatrix}r & -s\\\epsilon s^\ast & r^\ast\end{bmatrix},
        \quad
        s := px_+\lambda
  \end{equation}
  extend holomorphically to $\Sone$ and
  satisfy $P_0=P_0^\ast$ and $P_1=P_1^\ast$
  in the sense of~\eqref{eq:unitary}.
\end{proof}

\subsection{Scalar Birkhoff factorization}
\label{sec:birkhoff}
To unitarize the trinoid monodromy via lemma~\ref{lem:trinoid-monodromy}
we need a more detailed analysis of the scalar Birkhoff factorization.
With
$\Pone = \disk_+ \sqcup \Sone \sqcup \disk_-$,
\begin{equation}
  \disk_+:=\{\lambda\in\bbC\suchthat\abs{\lambda}<1\}
  \quad \disk_-:=\{\lambda\in\bbC\suchthat\abs{\lambda}>1\}\cup\{\infty\}
\end{equation}
let
\begin{subequations}
  \begin{align}
    \Lambda &= \{\text{real analytic loops on $\Sone\to\bbC^\ast$}\}\\
    \Lambda_+ &= \{f\in\Lambda\suchthat \text{$f$ is the boundary of a holomorphic
      map $\disk_+\to\bbC^\ast$}\}\\
    \Lambda_- &= \{f\in\Lambda\suchthat \text{$f$ is the boundary of a holomorphic
      map $\disk_-\to\bbC^\ast$}\}.
  \end{align}
\end{subequations}

1.
By~\cite{Pressley_Segal_1986}, every loop $f\in\Lambda$ has a unique Birkhoff factorization
\begin{subequations}
  \begin{gather}
    f = c \lambda^n f_- f_+,\quad n\in\bbZ,\quad c\in\bbC^\ast,\quad\\
    f_+\in\Lambda_+,\quad f_-\in\Lambda_-,\quad
    f_+(0)=1,\quad f_-(\infty)=1.
  \end{gather}
\end{subequations}

Let star be as in~\eqref{eq:scalar-star} for either choice of $\delta\in\{\pm 1\}$.

2.
If $f\in\Lambda$ satisfies $f=f^\ast$, then $f$ has a unique Birkhoff factorization
\begin{equation}
  f = \epsilon f_+^\ast f_+,\quad \epsilon\in\{\pm 1\},\quad
  f_+\in\Lambda_+,\quad f_+(0)\in\bbR_+.
\end{equation}
This defines an homomorphism
\begin{equation}
  \label{eq:sign0}
  \sign:\{f\in\Lambda\suchthat f=f^\ast\}\to\{\pm 1\},\quad
  f\mapsto \epsilon.
\end{equation}

A \emph{meromorphic loop} on $\Sone$ is meromorphic in a neighborhood of $\Sone$.
Let $\loopgroup$ be the subgroup of meromorphic loops on $\Sone$
\begin{equation}
  \loopgroup = \{\text{$f\suchthat f^\ast = f$ and $\Div f$ even}\},
\end{equation}
where $\Div f$ even means that the order of each zero and pole of $f$ on $\Sone$ is even.

3.
Every $f\in\loopgroup$ has a unique Birkhoff factorization
\begin{equation}
  f = \epsilon f_+^\ast f_+,\quad \epsilon\in\{\pm 1\},\quad
  \Div f_+ = \half\Div f
\end{equation}
where $f_+$ is the boundary of a holomorphic map $\disk_+\to\bbC^\ast$.
This extends the homomorphism~\eqref{eq:sign0}
to an homomorphism
\begin{equation}
  \sign:\loopgroup\to\{\pm 1\},\quad
  f\mapsto \epsilon.
\end{equation}

4.
For every meromorphic loop $f$ on $\Sone$ with $f=f^\ast$
\begin{equation}
  \label{eq:sign-square}
  \sign[f^2] = 1.
\end{equation}

\subsection{The trace polynomial}
Given a monodromy representation $M_0,\,M_1,\,M_2$ with $M_0M_1M_2=\id$ on the
three-punctured sphere,
define the polynomial~\cite{Goldman_1988}
\begin{equation}
  \label{eq:phi}
  \varphi = 1 - t_0^2 - t_1^2 - t_2^2 + 2 t_0 t_1 t_2,\quad
  t_k = \half\tr(M_k),\quad k\in\{0,\,1,\,2\}.
\end{equation}
The trace polynomial vanishes precisely where the monodromy representation
is reducible. In the case of the involution $\lambda\mapsto 1/\ol{\lambda}$,
if the halftraces are in $(-1,\,1)$, then
the monodromy is $\matSU_2$-unitarizable if $\varphi > 0$ and
$\matSU_{11}$-unitarizable if $\varphi < 0$.

For the trinoid potential~\eqref{eq:trinoid-potential}
with quadratic residues $(q_0,\,q_1,\,q_2)\in\bbR^3$  of $Q$,
and $f$ as in~\eqref{eq:trinoid-f},
\begin{equation}
  \label{eq:t}
  t_k = \cos(2\pi \nu_k),\quad
  \nu_k = \half - \half\sqrt{1 + q_k \kappa},\quad
  \kappa = 4\lambda^{-1}f(\lambda),\quad k\in\{0,\,1,\,2\}.
\end{equation}
The function $\kappa:\Sone\to\bbC$ satisfies $\kappa^\ast =\kappa$.

The trace polynomial $\varphi:\Sone\to\bbC$ for the monodromies
of the trinoid potential satisfies
$\varphi^\ast = \varphi$.
Putting together~\eqref{eq:phi} and~\eqref{eq:t},
its series expansion in $\kappa$ at $\kappa=0$ is
\begin{equation}
  \label{eq:phi-series}
  \varphi(\kappa) = c \kappa^4 + \mathrm{O}(\kappa^5),\quad
  c = \frac{\pi^4}{64}(q_0+q_1+q_2)(-q_0+q_1+q_2)(q_0-q_1+q_2)(q_0+q_1-q_2).
\end{equation}

\begin{lemma}
  \label{lem:phi}
  Choose $(q_0,\,q_1,\,q_2)\in\bbR^3$ with $c\ne 0$.
  For $s$ in a small enough interval $(0,\,s_0)$,
  $\varphi_{sq}\in\loopgroup$ and $\sign[\varphi_{sq}]=\sign(c)$.
\end{lemma}

\begin{proof}
  From the series expansion~\eqref{eq:phi-series},
  \begin{equation}
    \abs{\varphi_q/\kappa^4-c} < \abs{c/2}
    \quad\text{for all}\quad
    \abs{\kappa} < \kappa_0.
  \end{equation}
  Let $m = 1+\max_{\lambda\in\Sone}\abs{\kappa(\lambda)}$.
  Then for all $\abs{\kappa}<m$ and $s\in(0,\,\kappa_0/m)$,
  we have $\abs{s\kappa} < \kappa_0$, so
  \begin{equation}
    \abs{\varphi_{sq}(\kappa)/\kappa^4-c}
    =\abs{\varphi_{q}(s\kappa)/\kappa^4-c}
    < \abs{c/2}
  \end{equation}
  Hence $\varphi_{sq}/\kappa^4$ maps $\Sone$ into
  the disk centered at $c$ with radius $\abs{c/2}$.
  This implies that $\varphi_{sq}$ has no zeros
  on $\Sone$, except at the zeros of $\kappa$ in the case $\kappa$ has zeros on $\Sone$.
  Since these zeros are of order $4$, then $\varphi\in\loopgroup$.

  Since $\sign(c)\varphi_{sq}/\kappa^4$ takes values in the open right halfplane,
  it has a single-valued square root $f:\Sone\to\bbC^\ast$ satisfying $f^\ast = f$.
  By~\eqref{eq:sign-square},
  $\sign[\varphi_{sq}]=\sign(c)\sign[\kappa^4 f^2] = \sign(c)$.
\end{proof}

In the case of isosceles trinoids, for which $(q_0,\,q_1,\,q_2) = (a,\,a,\,b)$
with $a$, $b$ as in~\eqref{eq:trinoid-hopf},
then $c = \frac{\pi^4}{64}b^2(4a^2-b^2)$, so
$\sign[\varphi]=+1$ in a subregion of $\{\abs{b} < 2\abs{a}\}$.
and
$\sign[\varphi]=-1$ in a subregion of $\{\abs{b} > 2\abs{a}\}$.

\subsection{Trinoids}

\begin{theorem}
  \label{thm:trinoid}
  For each of the four symmetric spaces tabulated
  in~\eqref{eq:trinoid-table} there exists a two real parameter family
  of conformal CMC isosceles trinoid potentials satisfying the closing
  conditions with prescribed mean curvature $H$.
\end{theorem}

\begin{proof}
  Choose a symmetric space,
  $\epsilon\in\{\pm 1\}$ and $\delta\in\{\pm 1\}$
  as tabulated in~\eqref{eq:trinoid-table},
  and a isosceles trinoid potential~\eqref{eq:trinoid-potential}--\eqref{eq:trinoid-hopf}.

  By lemma~\ref{lem:phi}, there exists a subregion of $\{(a,\,b)\in\bbR^2\}$
  with $(a,\,b)$ as in~\eqref{eq:trinoid-hopf},
  in which $\varphi\in \loopgroup$ and $\sign[\varphi]=\delta\epsilon$.

  Let $M_0$ and $M_1$ be the monodromies
  of the meromorphic frame $\Phi$ based at $\Phi(0)=\id$
  as in~\eqref{eq:trinoid-monodromy}.
  The respective halftraces $t_0,\,t_1,\,t_2$ of $M_0,\,M_1,\,M_2 = M_1^{-1}M_0^{-1}$
  are
  \begin{equation}
    t_0 = t_1 = \half(r+r^\ast)
    \quad\text{and}\quad
    t_2 = \half(r^2 - 2pq + {r^\ast}^2)
  \end{equation}
  so the trace polynomial~\eqref{eq:phi} is
  \begin{equation}
    \label{eq:phi-in-terms-of-monodromy}
    \varphi = {(\imi(r-r^\ast))}^2pq.
  \end{equation}

  Since $\varphi\in \loopgroup$ and ${(\imi(r-r^\ast))}^2\in \loopgroup$, then $pq\in \loopgroup$.
  Since $p^\ast p\in \loopgroup$ then $q/p^\ast = (pq)/(p^\ast p)\in \loopgroup$.
  By~\eqref{eq:sign-square}
  $\sign[{(\imi(r-r^\ast))}^2]=1$
  and $\sign[p^\ast p] = 1$, so
  \begin{subequations}
    \begin{align}
      \sign[q/p^\ast] &=
      \sign[(pq)/(p^\ast p)] =
      \sign[pq]\sign[p^\ast p] = \sign[pq]\\
      &=
      \sign[{(\imi(r-r^\ast))}^2]\sign[pq] = \sign[\varphi] = \delta\epsilon.
    \end{align}
  \end{subequations}

  By the definition of $\sign$, $p/q^\ast$ has a factorization
  $p/q^\ast = \delta\epsilon x_+^\ast x_+$, where $x_+:\disk_+\to\bbC^\ast$.
  By lemma~\ref{lem:trinoid-monodromy}
  there exists a unitarizer $X$ of the monodromy in the sense of~\eqref{eq:unitary}.

  By theorem~\ref{thm:gwr} the immersion induced by $X\Phi$
  via $r$-Iwasawa factorization
  is conformal and CMC, and by remark~\ref{rem:closing} it closes around
  each of its three Delaunay ends.
\end{proof}

\begin{remark}
  \mbox{}

  1. In the case of $\spaceS$
  theorem~\ref{thm:trinoid} reconstructs a two-dimensional subfamily
  of the three-dimensional family
  of trinoids constructed in~\cite{Schmitt_Kilian_Kobayashi_Rossman_2007}
  without the isosceles constraint.

  2. In the other three symmetric spaces,
  where the Iwasawa factorization can leave the big cell,
  the domain of the trinoids constructed by
  theorem~\ref{thm:trinoid} is not known,
  but numerical experiments indicate
  that in general the whole punctured Riemann sphere is mapped into the light cone,
  intersecting the ideal boundary along curves
  (figures~\ref{fig:noid-h3},~\ref{fig:noid-h3-iso},~\ref{fig:trinoid-h3-halfspace} and~\ref{fig:trinoid}).

  3. Noids with more than three ends can be constructed
  using a cyclic branched cover of the Riemann sphere
  (figures~\ref{fig:noid-h3} and~\ref{fig:noid-h3-iso}).
\end{remark}

\subsection{Dressing}

\begin{figure}[b]
  \centering
  \includegraphics[width=0.45\textwidth]{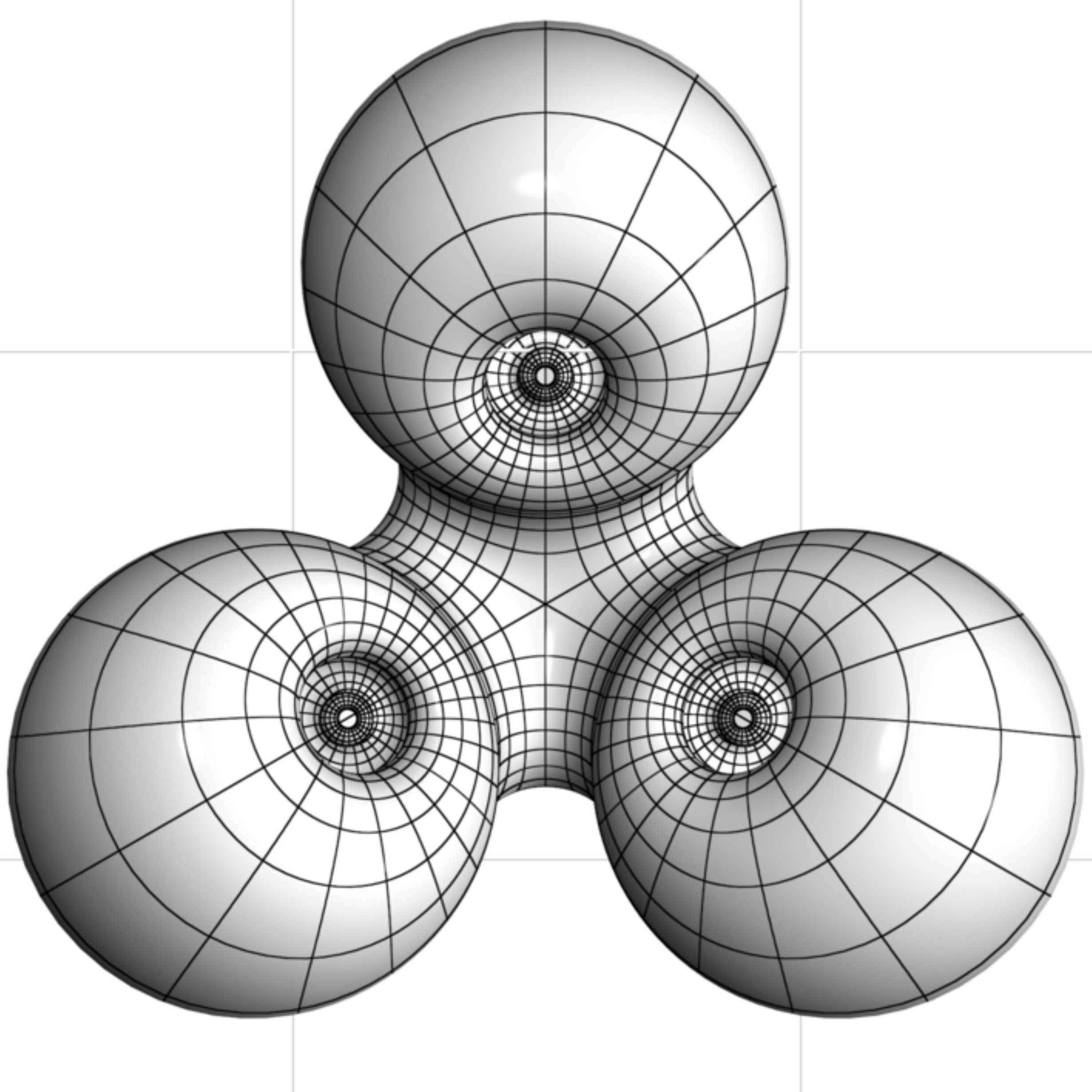}
  \includegraphics[width=0.45\textwidth]{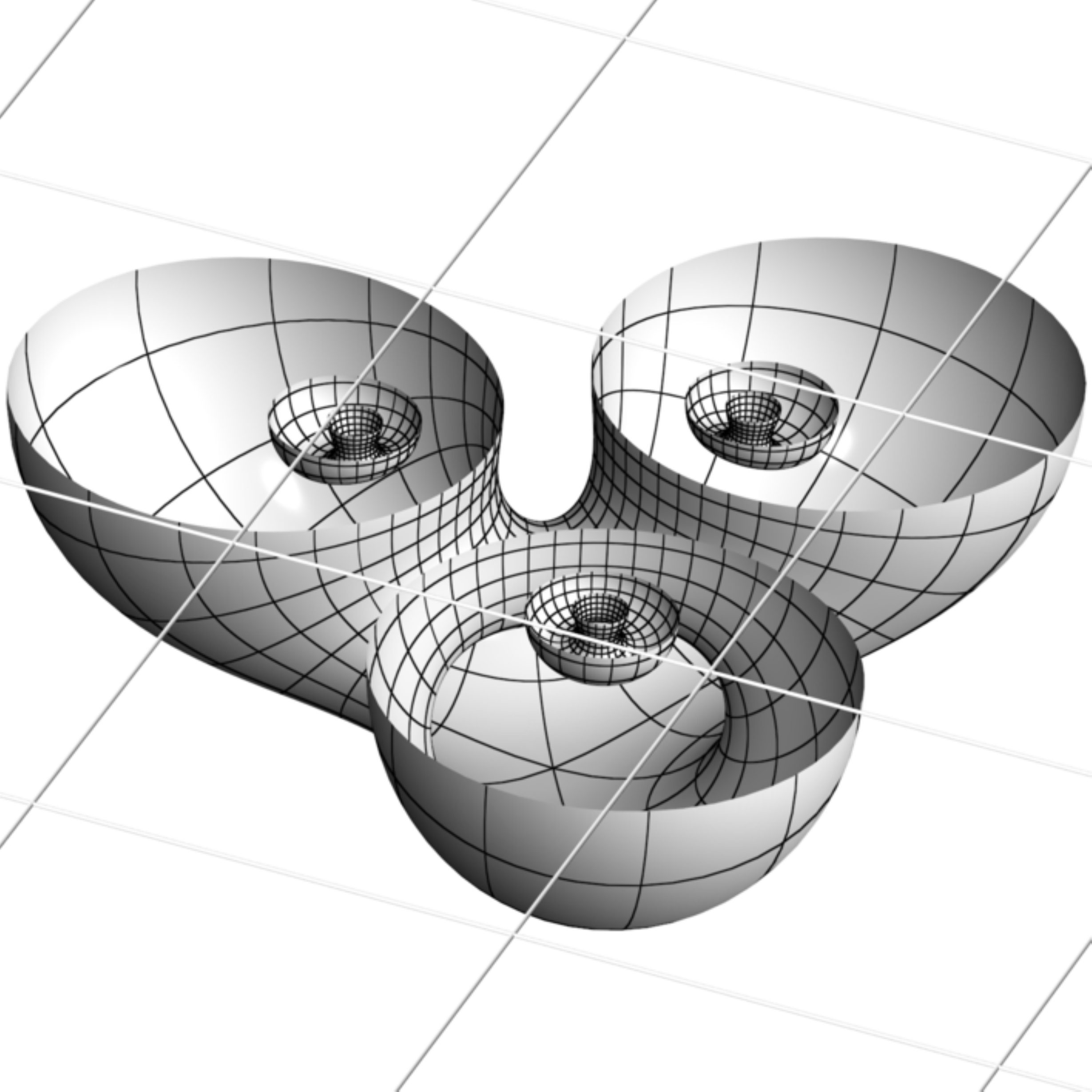}
  \caption{\small
    Two views of an equilateral minimal trinoid in
    $\spaceH\cup\Stwo\cup\spaceH$ stereographically projected to $\bbR^3$ via
    the Poincar\'e halfspace model.  The cutaway view (right) shows the
    intersection of the surface with the ideal boundary along
    nested topological circles.}
  \label{fig:trinoid-h3-halfspace}
\end{figure}

By theorem~\ref{thm:trinoid} we can construct equilateral
trinoids in $\spaceDS$ (figure~\ref{fig:trinoid-ds3}),
and isosceles non-equilateral trinoids in
$\spaceH$ (figure~\ref{fig:noid-h3-iso}).
In this section we construct equilateral trinoids in
$\spaceH$ (figure~\ref{fig:noid-h3}) by dressing equilateral trinoids
in $\spaceDS$. The dressing action interchanges the real
forms for $\spaceDS$ and $\spaceH$.

More explicitly, the \emph{dressing action} of a loop $g$ (defined on
an a circle of radius $r\in(0,\,1]$) on a frame $F$, written $g\dress F$,
is by definition the $r$-unitary factor of $gF$ of its $r$-Iwasawa factorization.
In analogy to~\cite{Terng_Uhlenbeck_2000} we take $g$ to be
a diagonal \emph{simple factor}
\begin{equation}
  \label{eq:simple-factor}
    g = \diag( p^{1/2},\, p^{-1/2} ),\quad
  p = \frac{\lambda - \mu}{\ol{\mu}\lambda + 1},\quad
  p^\ast = -\frac{1}{p}
\end{equation}
on an $r$-circle, $r < \abs{\mu}$.

\begin{lemma}
  \label{lem:simple-factor-dressing-swap}
  Diagonal simple factor dressing $F\mapsto g\dress F$
  interchanges the real form for $\spaceDS$ with the real form for $\spaceH$.
\end{lemma}

\begin{proof}
The dressing action of the simple factor $g$ on a frame $F$ in the real form for
$\spaceDS$
is computed explicitly and algebraically as
\begin{equation}
  g\dress F = g F k^{-1} g^{-1},\quad
  k = \tfrac{1}{\sqrt{{\abs{u}}^2 - {\abs{v}}^2}}
  \big[\begin{smallmatrix}u & \ol{v}\\v & \ol{u}\end{smallmatrix}\big],\quad
  \big[\begin{smallmatrix}u\\v\end{smallmatrix}\big] = F^{-1}(\mu)\ell,\quad\
  \ell = \big[\begin{smallmatrix}1\\0\end{smallmatrix}\big].
\end{equation}
We have
  \begin{equation}
    Fk^{-1} = \begin{bmatrix}s & t\\t^\ast & s^\ast\end{bmatrix}
      \quad\Longrightarrow\quad
      g\dress F = g F k^{-1} g^{-1} =
      \begin{bmatrix}s & t p\\t^\ast p^\ast & s^\ast\end{bmatrix}.
  \end{equation}
  Hence $g\dress F$ is in the real form for $\spaceH$.
  The proof for the other direction is the same except for a sign change.
\end{proof}

\begin{theorem}
  \label{thm:trinoid-h3}
There exists a real
$1$-parameter family of equilateral trinoids in $\spaceH$
for each mean curvature $\abs{H}<1$
(figure~\ref{fig:noid-h3}).
\end{theorem}

\begin{proof}
Let $F$ be a unitary frame for a trinoid in the $1$-parameter family
of equilateral trinoids in $\spaceDS$ with specified mean curvature
$\abs{H}<1$ (theorem~\ref{thm:trinoid}).
The dressed frame $g\dress F$ is in the real form for $\spaceH$
(lemma~\ref{lem:simple-factor-dressing-swap})
so it satisfies the intrinsic closing condition for $\spaceH$.

To satisfy the extrinsic closing condition we choose the simple factor $g$
as follows.
The subset of the $\lambda$ plane along which the
monodromy group of $F$ is reducible is the zero set of the trace polynomial
$\varphi$~\eqref{eq:phi}.
With $t$ the halftrace of the monodromy of $F$ around each end,
we have $\varphi= {(1-t)}^2(1+2t)$ with zero set
\begin{equation}
  \big\{
  \mu\in\bbC^\ast\suchthat \half -\half \sqrt{ 1 + 4 q \mu^{-1}f(\mu)} \in\bbZ/3
    \big\},\quad \text{$f$ as in~\eqref{eq:trinoid-f2}},
    \end{equation}
a discrete subset of $\bbC^\ast$
which accumulates at $0$ and $\infty$.
Using~\eqref{eq:trinoid-monodromy} and~\eqref{eq:phi-in-terms-of-monodromy},
$\mu$ can be found
in this zero set away from the \sym points
such that $\ell$ is a common eigenvalue
of the monodromy group. This choice of $\mu$ implies that
the monodromy of $g\dress F$ is $gMg^{-1}$,
where $M$ is a monodromy of $F$
(in analogy to~\cite{Kilian_Schmitt_Sterling_2004} for noids in $\bbR^3$).
Hence
$g\dress F$ satisfies the extrinsic closing conditions for $\spaceH$,
so it is the frame for an equilateral trinoid in $\spaceH$.
\end{proof}

\begin{remark}
  \label{rem:kobayashi}
  \mbox{}

  The \DPW construction of trinoids (theorem~\ref{thm:trinoid})
  is as follows:
  \begin{enumerate}
  \item
    In terms of a potential~\eqref{eq:trinoid-potential}
    on the three-punctured sphere,
    compute a dressing $C$ which unitarizes the
    monodromy group of the corresponding holomorphic frame $\Phi$
    based at $\Phi(z_0)=\id$.
  \item
    Show that the dressed holomorphic frame
    $C\Phi$ is in the big cell of the Iwasawa decomposition
    in some subregion of the domain.
  \item
    Construct a CMC immersion of that subregion into $\spaceH$
    via the \sym formula~\eqref{eq:sym} applied
    to the unitary Iwasawa factor of $C\Phi$.
  \end{enumerate}

  Following this program with gauge equivalent trinoid potentials,
  Kobayashi in \cite[Theorem 4.1]{Kobayashi_2010}
  gives an incomplete construction of equilateral trinoids in $\spaceH$
  with mean curvature $\abs{H}< 1$, failing to address step (2).
  Our theorem~\ref{thm:trinoid-h3} fills this gap
  by constructing a dressing $C$ in the loop group $\Lambda_+\matSL_2(\bbC)$.
  Thus $C\Phi$ is in the big cell of the Iwasawa decomposition
  at, and hence in a neighborhood of the basepoint $z_0$. Moreover, it follows from Theorem \ref{thm:main}  that for small necksizes $C\Phi$ is in the big cell of the Iwasawa decomposition on the 3-holed sphere obtained by removing  3 discs around the punctures of the 3-punctured sphere.

  Dorfmeister, Inoguchi and Kobayashi further state without proof
  in~\cite[\S 10.5]{Dorfmeister_Inoguchi_Kobayashi_2014}
  that the theorem of Kobayashi mentioned
  above constructs CMC immersions of the three-punctured sphere into $\spaceH$.
  On the contrary, in light of the numerical evidence of
  figures~\ref{fig:noid-h3},~\ref{fig:noid-h3-iso},~\ref{fig:trinoid-h3-halfspace}
  and~\ref{fig:trinoid}, we conjecture that
  in analogy to 2-noids (Delaunay surfaces), the incompletely constructed
  trinoids in \cite{Kobayashi_2010}
  map the three-punctured sphere not into
  $\spaceH$ but into $\spaceH\cup\Stwo\cup\spaceH$.
\end{remark}

\providecommand{\bysame}{\leavevmode\hbox to3em{\hrulefill}\thinspace}
\providecommand{\MR}{\relax\ifhmode\unskip\space\fi MR }
\providecommand{\MRhref}[2]{%
  \href{http://www.ams.org/mathscinet-getitem?mr=#1}{#2}
}
\providecommand{\href}[2]{#2}

\end{document}